\documentclass[a4paper,12pt]{article}

\usepackage[margin=3cm,footskip=1cm]{geometry}

\usepackage[T1]{fontenc}

\usepackage{amsmath,amssymb,amsthm,bm,mathtools,xspace,booktabs,url}
\usepackage{graphicx,etoolbox}

\usepackage[shortlabels]{enumitem}
\setlist{
  listparindent=\parindent,
  parsep=0pt,
}

\usepackage[round]{natbib}

\AtBeginEnvironment{thebibliography}{%
  \small
  \setlength\itemsep{0.75em plus 0.5em minus 0.75em}%
}

\usepackage{caption}
\usepackage[raggedright]{titlesec}

\numberwithin{equation}{section} 

\theoremstyle{plain} 

\theoremstyle{definition} 

\usepackage{authblk}

\makeatletter
\newcommand\CorrespondingAuthor[1]{%
  \begingroup%
  \def\@makefnmark{}%
  \footnotetext{Corresponding author: #1}%
  \endgroup%
}
\makeatother

\makeatletter
\renewenvironment{abstract}{%
  \small%
  \providecommand\keywords{%
    \par\medskip\noindent\textit{Keywords:}\xspace}%
  \begin{center}%
    \bfseries \abstractname\vspace{-.5em}\vspace{\z@}%
  \end{center}%
  \quote%
}{\endquote}
\makeatother

\usepackage[above,below,section]{placeins}

\usepackage[load-configurations=abbreviations]{siunitx}

\usepackage{lipsum}

\usepackage{tikz}
\usetikzlibrary{arrows}

\newcommand{\EE}{{\mathbb E}}
\newcommand{\RR}{{\mathbb R}}

\newcommand{\bx}{{\boldsymbol x}}
\newcommand{\bX}{{\boldsymbol X}}
\newcommand{\by}{{\boldsymbol y}}
\newcommand{\bY}{{\boldsymbol Y}}
\newcommand{\bth}{\boldsymbol{\theta}}
\newcommand{\dee}[1]{\, {\rm d}{#1}}

\newcommand{\ra}{\rightarrow}

\newcommand{\ddl}{{d}_L^\ra}
\newcommand{\an}{\mathrm{an}}
\newcommand{\pa}{\mathrm{pa}}
\newcommand{\ch}{\mathrm{ch}}
\newcommand{\de}{\mathrm{de}}
\newcommand{\lik}{{\cal L}}

\newtheorem{prop}{Proposition}

\begin{document}

\title{Point processes on directed linear network}

\author{Jakob G. Rasmussen}

\author{Heidi S. Christensen}

\affil{Department of Mathematical Sciences, Aalborg University}

%
%

\date{}


\maketitle

\begin{abstract}

  In this paper we consider point processes specified on directed
  linear networks, i.e.\ linear networks with associated
  directions. We adapt the so-called conditional intensity function
  used for specifying point processes on the time line to the setting
  of directed linear networks. For models specified by such a
  conditional intensity function, we derive an explicit expression for
  the likelihood function, specify two simulation algorithms (the
  inverse method and Ogata's modified thinning algorithm), and
  consider methods for model checking through the use of residuals. We
  also extend the results and methods to the case of a marked point
  process on a directed linear network. Furthermore, we consider
  specific classes of point process models on directed linear networks
  (Poisson processes, Hawkes processes, non-linear Hawkes processes,
  self-correcting processes, and marked Hawkes processes), all adapted
  from well-known models in the temporal setting. Finally, we apply
  the results and methods to analyse simulated and neurological data.

  \keywords Conditional intensity, dendrite network, Hawkes process,
  self-correcting process
\end{abstract}

\section{Introduction}
\label{intro}
Point processes on linear networks are important for modelling events or objects on a real network, such as a road or river network. In recent years there have been a fair amount of papers on functional summary statistics and models for point processes specified on linear networks \citep[see][]{okabe-01,ang-etal-12,baddeley-etal-14,mcswiggan-etal-16,baddeley-etal-17,rakshit-etal-17}. Specifically, \cite{okabe-01} present a network analogue of Ripley's $K$-function of which \cite{ang-etal-12} suggest a correction that compensates for the geometry of the network, making it possible to compare $K$-functions for different networks directly. For these $K$-functions it is required that the point process is second-order pseudostationary, meaning that the intensity is constant and the pair correlation function depends only on the geodesic distance. However, \cite{baddeley-etal-17} discuss the difficulties of finding such point processes, and \cite{rakshit-etal-17} discuss using alternative distance metrics and present analogues of the $K$-function and pair correlation function wrt.\ these metrics. Further, \cite{baddeley-etal-14} present methods for analysing multitype point processes on networks, and \cite{mcswiggan-etal-16} address problems with existing kernel estimates of the intensity point processes and further develop a new kernel estimate eluding these problems. 

In the present paper we consider directed linear networks, i.e.\ networks consisting of line segments with an associated direction. Such directions appear naturally in some applications, while directions cannot be used, or at the very least are rather artificial, in other applications. For example, river networks have a natural direction following the flow of water, while the bidirectionality of (most) roads means that directed networks are not particularly useful as models for road networks. 
\cite{ver-hoef-etal-06}, \cite{garreta-etal-09}, and \cite{ver-hoef-etal-12} consider Gaussian processes and covariance functions on so-called stream networks, which are special cases of directed linear networks. In the present paper, however, we focus on point processes specified by a modified version of the conditional intensity function often used for point processes on the time line (see e.g.\ Chapter 7 in \cite{daley-vere-jones-03} for an introduction to these). 
On the time line, the conditional intensity is based on conditioning on the past, and for a directed linear network the directions enable us to modify the notion of past and thereby to extend the definition of a conditional intensity. 

There are many types of data suitable for modelling by a point process specified by a conditional intensity function on a directed linear network. One example is the locations of spines  along a dendritic tree, where we can introduce directions going from the root of the tree towards the leaves of the network.  Spines play a role in e.g.\ memory storage, and changes in the spine distribution and shape have been linked to neurological diseases \citep{irwin-etal-01}. Only a few studies \citep{jammalamadaka-etal-13,baddeley-etal-14} model the distribution of spines using point processes on (undirected) linear networks. Further, these studies only consider the Poisson process and the multitype Poisson process when spines are classified into types depending on their shape. 

The outline of the paper is as follows: In Section~\ref{sec.PPLN} we first define directed linear networks and a number of related concepts, and next what we mean by a conditional intensity function on a directed linear network. In Section~\ref{sec.likfunc} we derive the likelihood function for a point process specified by such a conditional intensity function, and in Section~\ref{sec.sim} we consider two simulation algorithms. In Section~\ref{sec:residual_analysis} we discuss a method for model checking based on residuals. In Section~\ref{sec.models}, using the conditional intensity function, we define a number of models for point processes on directed linear networks all inspired from similar models on the time line. In Section~\ref{sec.data} we use the presented models and methods to analyse simulated datasets and a real dataset consisting of spines along a dendritic tree. Finally, we round off the paper by considering possible extensions and future research directions in Section~\ref{sec.ext}. 

\section{Point processes on directed linear networks}\label{sec.PPLN}

\subsection{Directed linear networks}\label{sec.DLN}

Let $L_i\subseteq\RR^d$, $i = 1,\dots,N$,  denote an open line segment of finite length $|L_i|$, where $\RR^d$ denotes the $d$-dimensional Euclidean space for $d\geq2$. A direction is associated to each line segment $L_i$, where we denote the endpoints of $L_i$ by $\underline{e}_{i}, \overline{e}_{i}\in\RR^d$ such that the direction goes from $\underline{e}_{i}$ to $\overline{e}_{i}$. Furthermore, we assume that the line segments do not overlap
(but they are of course allowed to join at their endpoints in order to form a network). Any point that is the endpoint of at least one line segment is called a vertex. The line segment $L_i$ is an outgoing line segment from the vertex at $\underline{e}_{i}$ and ingoing at the vertex at $\overline{e}_{i}$. The $i$th line segment can conveniently be represented by the parametrisation
\[
u_i(t) = \underline{e}_{i} + t\frac{\overline{e}_{i}-\underline{e}_{i}}{\|\overline{e}_{i}-\underline{e}_{i}\|},
\]
where $t\in(0,|L_i|)$ and $\|\cdot\|$ denotes Euclidean
distance. Occasionally we will consider only parts of a line segment, where $L_i(t_1,t_2)$ denotes the set $\{u_i(t):\	t\in (t_1, t_2)\}$ for $0\leq t_1 \leq t_2 \leq |L_i|$. See Figure~\ref{Fig1} for illustrations of the above concepts.
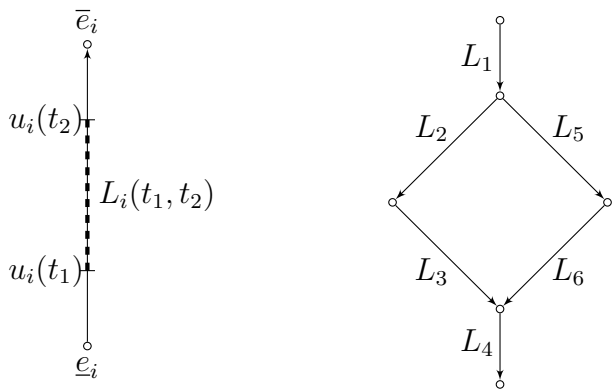
\begin{figure}
	\centering
	\begin{tikzpicture}
	\tikzset{vertex/.style = {shape=circle, draw, inner sep=1pt}}
	\tikzset{edge/.style = {->,> = latex'}}
	\node[vertex] (a) at (0, 0) {};
	\node[vertex] (b) at (0, 4) {};	
	\draw[edge] (a) -- (b);	
	\draw[dashed, ultra thick] (0, 1) -- (0, 3);
	\node[below] at (a) {$\underline{e}_{i}$};
	\node[above] at (b) {$\overline{e}_{i}$};
	\node[right] at (0, 2) {$L_i(t_1, t_2)$};
	\draw (-0.1, 1) -- (0.1, 1) node[left] {$u_i(t_1)$};
	\draw (-0.1, 3) -- (0.1, 3) node[left] {$u_i(t_2)$};
	\end{tikzpicture}\hspace{2cm}
	\begin{tikzpicture}
	\tikzset{vertex/.style = {shape=circle, draw, inner sep=1pt}}
	\tikzset{edge/.style = {->,> = latex'}}
	\node[vertex] (a) at  (0, 4.828427) {};
	\node[vertex] (b) at  (0, 3.828427) {};
	\node[vertex] (c) at  (-1.414214, 2.414214) {};
	\node[vertex] (d) at  (1.414214, 2.414214) {};
	\node[vertex] (e) at (0, 1) {};
	\node[vertex] (f) at (0,0) {};
	\draw[edge] (a) -- (b);
	\draw[edge] (b) -- (c);
	\draw[edge] (b) -- (d);
	\draw[edge] (c) -- (e);
	\draw[edge] (d) -- (e);
	\draw[edge] (e) -- (f);
	\node at (-0.3, 4.328427) {$L_1$};
	\node at (-0.9, 3.4) {$L_2$};
	\node at (-0.9, 1.5) {$L_3$};
	\node at (0.9, 3.4) {$L_5$};
	\node at (0.9, 1.5) {$L_6$};
	\node at (-0.3, 0.55) {$L_4$};
	\end{tikzpicture}
	\caption{Left: a directed line segment $L_i$ with endpoints $\underline{e}_i$ and $\overline{e}_{i}$ and a partial line segment $L_i(t_1, t_2)$ starting in $u_i(t_1)$ and ending in $u_i(t_2)$. Right: a DALN consisting of the line segments $L_1,  \dots, L_6$}
	\label{Fig1}
\end{figure}

The union of the line segments is denoted by $L^\cup = \bigcup_{i=1}^N L_i$, while the set of line segments is denoted by $L = \{L_i:i=1,\ldots,N\}$. The terminology directed linear network may refer to either $L$ or $L^\cup$ depending on the context.

We have used open line segments to build the directed linear
network in order to avoid endpoints being part of multiple
line segments. Obviously, when we later define point processes on
networks, this means that there cannot be any points exactly at the
vertices, but since we will anyway consider only point processes with
a diffuse measure, such points would occur with probability zero. 

\subsection{Directed paths and partial orders}\label{sec:DPPO}

We define a directed path of line segments going from $L_{i}\in L$ to $L_{j}\in L$, where $i\not= j$, in the following way: Let $I\subseteq\{1,\ldots,N\}$ be indices for a subset of $L$ with cardinality $|I|$ and $i, j \in I$. Then $(L_{\psi(1)}, \dots, L_{\psi(|I|)})$ is called a directed path from $L_{i}$ to $L_{j}$ if $\psi:\{1, \dots, |I|\}\rightarrow I$ is a bijection such that $\psi(1)=i$, $\psi(|I|)=j$, and $\underline{e}_{\psi(k)} = \overline{e}_{\psi(k-1)}$ for $k = 2,\ldots,|I|$. In other words, you can get from any point in $L_{i}$ to any point in $L_{j}$ following the directions of the line segments in the directed path. If at least one such directed path exists, we
write $L_{i}\ra L_{j}$. For the DALN in Figure~\ref{Fig1},  one possible directed path from $L_1$ to $L_4$ is $(L_1, L_5, L_6, L_4)$, where $I = \{1, 4, 5, 6\}$ and $\psi$ is specified by $\psi(1) = 1$, $\psi(2) = 5$, $\psi(3) = 6$, and $\psi(4) = 4$.

We extend the notion of a directed path to any pair of points in the
network: Let $u = u_i(t_1)\in L_{i}$ and $v = u_j(t_2) \in L_{j}$ for $L_{i},L_{j}\in L$ (not necessarily distinct), and let
\[
u \ra v \qquad\text{if}\enspace 
\begin{cases}
L_{i} \ra L_{j} & \text{for $i\not=j$},\\
t_1 < t_2 & \text{for $i=j$}.
\end{cases}
\]
A path from $u$ to $v$ is denoted by $p_{u\ra v}$ and consists of 
\[
(L_{\psi(1)}(t_1,|L_{\psi(1)}|), L_{\psi(2)}, \ldots, L_{\psi(|I|-1)}, L_{\psi(|I|)}(0,t_2)),
\]
where the first and last line segments have been restricted to points strictly between $u$ and $v$.
The length of a directed path $p_{u\ra v}$ is the sum of the
lengths of all line segments on the directed path (with the end line
segments restricted as above) and will be denoted $|p_{u\ra v}|$. 
Further, we let $P_{u\ra v}$ denote the set of all directed paths from $u$ to $v$.
If $u\ra v$, the length of the
shortest directed path from  $u$ to $v$ is denoted by
$\ddl(u,v)$, and if $u\not\ra v$ then we let
$\ddl(u,v)=\infty$ (except if $u=v$ in which case
$\ddl(u,v)=0$). Note that $\ddl(\cdot, \cdot)$ is a metric except that
it is not symmetric (i.e.\ it is a quasi-metric).

We restrict attention to a particular class of directed linear networks satisfying that if $u\ra v$ for any $u, v \in L^\cup$, then $v\not\ra u$; that is, there are no directed loops in the network. Such a network is called a directed acyclic linear network (DALN); two examples of DALNs are shown in Figure~\ref{Fig2}. Most results in this paper depend on this assumption. For a DALN, the relation $\ra$ is a strict partial order (i.e.\ it is non-reflexive, anti-symmetric and transitive) either on $L$ (when considering line segments) or on $L^\cup$ (when considering points). Furthermore, we often consider special cases of DALNs for which we use standard terminology from graph theory: A directed linear network without loops (disregarding directions) is called a tree, and a tree where all directions go either away from or towards a single vertex (called the root) is called a rooted tree (or an out-tree when the directions are away from the root and in-tree when they are towards the root). The right panel of Figure~\ref{Fig2} shows an example of an out-tree, where the root is the vertex at the bottom.
\begin{figure}
	\centering
	\begin{tikzpicture}[scale = 0.8]
	\tikzset{vertex/.style = {shape=circle, draw, inner sep=1pt}}
	\tikzset{edge/.style = {->,> = latex'}}
	\node[vertex] (a) at  (0, 6.828427) {};
	\node[vertex] (b) at  (0, 3.828427) {};
	\node[vertex] (c) at  (-1.414214, 2.414214) {};
	\node[vertex] (d) at  (2.645751, 2.414214) {};
	\node[vertex] (e) at (0, 1) {};
	\node[vertex] (f) at (0,0) {};
	\draw[edge] (a) to (b);
	\draw[edge] (b) to (c);
	\draw[edge] (b) to (d);
	\draw[edge] (c) to (e);
	\draw[edge] (d) to (e);
	\draw[edge] (e) to (f);
	\end{tikzpicture}\hspace{2cm}
	\begin{tikzpicture}[scale = 0.8]
	\tikzset{vertex/.style = {shape=circle, draw, inner sep=1pt}}
	\tikzset{edge/.style = {->,> = latex'}}
	\node[vertex] (a) at (0, 0) {};
	\node[vertex] (b) at (0, 2) {};
	\node[vertex] (c) at (1.414214, 3.414214) {};
	\node[vertex] (d) at (1.414214, 6.414214) {};
	\node[vertex] (e) at (4.242641, 6.242641) {};
	\node[vertex] (f) at (-2.121320, 4.121320) {};
	
	\draw[edge] (a) to (b);
	\draw[edge] (b) to (c);
	\draw[edge] (c) to (d);
	\draw[edge] (c) to (e);
	\draw[edge] (b) to (f);
	\end{tikzpicture}
	\caption{Two examples of DALNs. These DALNs are also used in the simulation study in Section~\ref{sec.simdata}}
	\label{Fig2}
\end{figure}
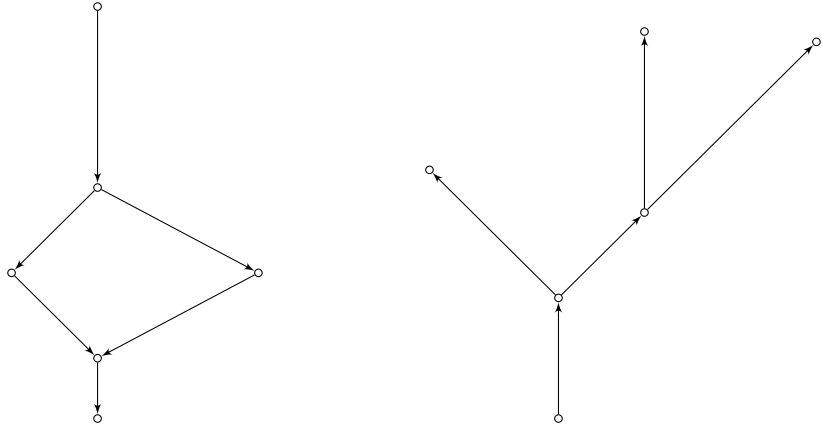

The relation $\ra$ induces a (not necessarily unique) order on all line
segments in a DALN. Let $\omega : \{1,\ldots,N\} \rightarrow \{1,\ldots,N\}$ be a bijection such that
$L_{\omega(i)} \ra L_{\omega(j)}$ whenever $i<j$, i.e.\ we get a new ordering of the line segments $L_{\omega(1)},\ldots,L_{\omega(N)}$ that follows the partial order $\ra$ whenever it applies to a pair of line segments. 
Denote the
set of all such bijections by $\Omega$. A bijection
$\omega\in\Omega$ can be obtained by picking out an arbitrary line
segment, then go against the directions until we reach a line segment
$L_{i}$ where no $L_{j} \ra L_{i}$ for $j\in\{1,\ldots,N\}$
(the existence of such a line segment follows since $L$ is finite and
$\ra$ is a strict partial order). We then let $\omega(1) = i$, such that $L_{\omega(1)}=L_{i}$ is the first line segment in the new ordering of the line segments, and
continue iteratively by considering the network $L\backslash\{L_{i}\}$. For the DALN in Figure~\ref{Fig1}, one choice of $\omega \in \Omega$ is $\omega(1) = 1$, $\omega(2) = 2$, $\omega(3) = 3$, $\omega(4) = 5$, $\omega(5) = 6$, and $\omega(6) = 4$.

While most results in this paper need a choice of
$\omega\in\Omega$, they do not depend on the actual choice of
$\omega$. Throughout the paper, it will be assumed that some choice of $\omega$ has been made whenever needed.

\subsection{Point patterns and point processes on linear networks}\label{sec.pointprocess}

A point pattern on a linear network is a finite set $\bx \subset L^\cup$, and a point process on a linear network is a stochastic process $\bX$ whose realisations are point patterns on the network. If we assume that the network is a DALN, we can use the order induced by $\ra$ to specify a point process by a
conditional intensity function as described in Section~\ref{sec.cif}.

To describe how points in a point pattern on a directed linear network are located relative to each other, we adopt further terms from graph theory. Let $\bx$ be a point pattern on a directed linear network and $x\in\bx$. Then the ancestors of $x$ are the set of points
\[
\an(x) = \{y\in \bx: y\ra x\},
\]
and the parents of $x$, $\pa(x)\subseteq\an(x)$, are the ancestors from which there exists at least one directed path to $x$ containing no other points of $\bx$. The descendants and children of $x$, denoted by $\de(x)$ and $\ch(x)$, are defined similarly, just reversing the direction.

\subsection{Conditional intensity functions on directed linear networks}\label{sec.cif}

Let $L_{\ra i} = \{L_j \in L: L_j \ra L_i\}$ denote the set of line segments with directed paths to $L_i \in L$, and similarly let $L_{\ra i}^\cup = \bigcup_{L_j \in L:L_j\ra L_i} L_j$ denote the union of these. We now define a point process on $L$ by defining a point process on each $L_i$ conditionally on the points in $L_{\ra i}$. 
First, recall that the conditional intensity function $\tilde\lambda^*$ of a temporal point process $\tilde\bX$ is given by
\[
\tilde \lambda^*(t) = \frac{\EE[N([t,t+dt])|\tilde\bX_t]}{dt},
\]
i.e.\ the mean number of points $N([t,t+dt])$ falling in an infinitesimally small time interval $[t,t+dt]$ starting at time $t$ conditional on the point process before time $t$ denoted by $\tilde\bX_t$ \citep[see e.g.][for more details on the conditional intensity function for temporal point processes]{daley-vere-jones-03}. Next, to adapt this concept to directed linear networks, we now let $\tilde \lambda^*$ denote the conditional intensity function of a temporal point process restricted to an interval $(0,|L_i|)$, and allow it to depend on a point pattern on $L_{\ra i}$. The resulting point pattern on $(0, |L_i|)$ is then mapped to $L_i$ by $u_i(t)$. We call this a point process with conditional intensity function $\lambda^*(u)$ for $u\in L_i$, where $\lambda^*(u) = \tilde \lambda^*(t)$ for $u = u_i(t)$.

To obtain a point process on $L^\cup$, we define a point process with
conditional intensity function $\lambda^*$ as above recursively on
$L_{\omega(1)}, \ldots, L_{\omega(N)}$, $\omega\in\Omega$. Note that
following this order ensures that whenever we define a point process
on a line segment~$L_i$, we have already defined it on $L_{\ra i}$, on
which we condition. Further, the conditional intensity does not depend
on the specific choice of permutation $\omega$.

As discussed in Section~\ref{sec.DLN}, our definition of a directed linear network does not include the vertices as a part of the network. However, in practice we may have datasets containing points located exactly on the vertices, e.g.\ if  the location of the points have been used as vertices when approximating the true network with line segments (this is indeed the case for the dendrite data considered in Section~\ref{sec.dendata}). Each of these points need to be allocated to a unique line segment such that the conditional intensity is correctly specified. How to do this depends on the nature of the network $L$. If $L$ is an out-tree where the root is of degree 1, we naturally allocate a point falling at the root to the starting point of the line segment starting in the root. Any other points falling at a vertex will be allocated the endpoint of the ingoing line segment of that vertex. Thus, the line segment going from the root has been extended to include both endpoints, while any other line segment $L_i$ include their second endpoint $\overline{e}_{i}$. Similar modifications can be made to other networks.

\subsection{Marks}\label{sec.marks}

Often additional information, referred to as marks, are associated with each point in a point pattern $\bx = \{x_1, \dots, x_n\}$. Assume that the marks belong to a space $\mathbb{M}$, which we call the mark space, and that each point $x_i \in \bx$  in the observed point pattern has an associated mark $m_i\in\mathbb{M}$. That is, a marked point pattern on a directed linear network is a finite set $\by = \{(x_1,m_1),\ldots,(x_{n},m_{n})\}\subset L^\cup\times\mathbb{M}$.

To define a marked point process, we let a mark associated with the point $u \in L$ follow a distribution, that may depend both on the location $u$ and the marked point process on $L^\cup_{\ra u}= L(t,|L_i|)\cup\bigcup_{L_j \ra L_i} L_j$ for $u = u_i(t)$. The conditional intensity can then be generalised to the marked case by
\[
\lambda^*(u,m) = \lambda^*(u)f^*(m|u),
\]
where $\lambda^*(u)$ is the intensity defined in
Section~\ref{sec.cif}, except that the star now means that it may
depend on marks of points on $L^\cup_{\ra u}$ in addition to the
points themselves, and $f^*({}\cdot{}\mid u)$ is the conditional density function of the mark given the points and marks on $L^\cup_{\ra u}$. Note that in the marked case, by a slight abuse of notation, we let $\lambda^*$ denote the conditional intensity function both depending on the point $u$ and mark $m$ as well as the conditional intensity function depending only on the point $u$.

\section{Likelihood function}\label{sec.likfunc}

We can obtain a closed form expression for the likelihood function for a point process on a DALN specified by a conditional intensity function.  Firstly, consider the measure $\lambda_1$, where $\lambda_1(A)$ is the total length of a measureable subset $A\subseteq L^\cup$. Furthermore, we use the notation $\bx_{(i)} = \bx\cap L_i$
and $\bx_{(\ra i)} = \bx\cap L_{\ra i}$. Finally, assume that the conditional intensity function depends on some parameter vector, say $\bth$.


\begin{prop}\label{prop.lik}
	Consider an unmarked point process $\bX$ on a DALN $L$ specified by a conditional intensity function $\lambda^*$ depending on a parameter vector $\bth$, and let $\bx$ be an observed point pattern dataset. Then the likelihood function is given by
	\[
          {\lik}(\bth; \bx) = \Bigl(\prod_{x \in \bx}
          \lambda^*(x;\bth)\Bigr)
          \exp\Bigl(-\int_L\lambda^*(u;\bth)\dee\lambda_1(u)\Bigr).
	\]	
	Similarly, if $\bY$ is a marked point process with conditional intensity function $\lambda^*$ depending on $\bth$, and $\by$ is an observed marked point pattern, then the likelihood function is given by
	\[
	\lik(\bth;\by) = \Bigl(\prod_{(x, m) \in \by} \lambda^*(x, m;\bth)\Bigr)
	\exp\Bigl(-\int_L\lambda^*(u;\bth)\dee\lambda_1(u)\Bigr).
	\]  
\end{prop}

\begin{proof}
  Consider first the unmarked case. Letting $\omega \in \Omega$, we
  split the likelihood into a product of density functions for the
  point pattern $\bx_{(\omega(i))}$ on each line segment conditional
  on the points patterns earlier in $\omega$ given by
  $\bx_{(\omega(1))},\ldots,\bx_{(\omega(i-1))}$. That is,
	\[
	\lik(\bth;\bx) = \prod_{i=1}^N
	f(\bx_{(\omega(i))}|\bx_{(\omega(1))},\ldots,\bx_{(\omega(i-1))};\bth).
	\]
	Since
        $\bx_{(\ra i)} \subseteq
        \bigcup_{j=1}^{i-1} \bx_{(\omega(j))}$ and
        $\bx_{(\omega(i))}|\bx_{(\omega(1))},\ldots,\bx_{(\omega(i-1))}$
        by construction depends only on $\bx_{(\ra i)}$, we get that
	\begin{equation}\label{eq.fact}
	\lik(\bth;\bx) = \prod_{i=1}^N f(\bx_{(\omega(i))}|\bx_{(\ra \omega (i))};\bth) =
	\prod_{i=1}^N f(\bx_{(i)}|\bx_{(\ra i)};\bth).
	\end{equation}
	By definition, a point process on $L_i$ specified conditionally on
	$\bx_{(\ra i)}$ by $\lambda^*(u;\bth)$ is equivalent to a temporal point
	process specified by $\tilde\lambda^*(t;\bth)$ on $(0,|L_i|)$, where
	$u=u_i(t)$. Thus, by Proposition 7.2.III in \cite{daley-vere-jones-03}, we get that
	\[
	f(\bx_{(i)}|\bx_{(\ra i)}) = \Bigl(\prod_{x \in \bx_{(i)}} \lambda^*(x;\bth)\Bigr)
	\exp\Bigl(-\int_{L_i}\lambda^*(u;\bth)\dee\lambda_1(u)\Bigr), 
	\]
	which together with \eqref{eq.fact} completes the proof for the unmarked case. The result for the marked case is proven in a similar manner, using Proposition~7.3.III instead of 7.2.III in \cite{daley-vere-jones-03}. 
\end{proof}

\section{Simulation}\label{sec.sim}

There are two general methods for simulating temporal point processes specified by a conditional intensity function: the inverse method and Ogata's modfied thinning algorithm. Both algorithms can be modified to work for a point process on a DALN by simulating the point process on one line segment at a time following the order given by
$\omega\in\Omega$. So we focus on specifying how to simulate the point process on a line segment $L_i$ conditional on the points already simulated on $L_{(\ra i)}$.

First, we consider the inverse method in the unmarked case. Let $u=u_i(t)$ for $t\in
(0,|L_i|)$, and let
\[
\Lambda^*(u) = \int_{L_i(0,t)} \lambda^*(v) \dee\lambda_1(v).
\]
In the inverse method, independent and identically distributed (IID) unit-rate exponential random variables are simulated and transformed into the appropriate points on $L_i$ by the inverse of $\Lambda^*$. More precisely, the algorithm is as follows:
\begin{enumerate}
	\item Let $j=0$
	\item Repeat:
	\begin{enumerate}
		\item Generate $Y_j\sim \mathrm{Exp}(1)$
		\item Find $t$ such that $\Lambda^*(u_i(t)) = \sum_{k=0}^{j} Y_k$ 
		\item If $t < |L_i|$, let $j = j +1$  and $x_j = u_i(t)$. Else end repeat loop
	\end{enumerate}
	\item Output $(x_1,\ldots,x_j)$ 
\end{enumerate}

Next, for Ogata's modified thinning algorithm in the unmarked case, we use that $\lambda^*(u)=\tilde\lambda^*(t)$ when $u=u_i(t)$ and assume that for any $t\in(0,|L_i|)$, there exist functions $L^*(t)>0$ and $M^*(t)\geq\tilde\lambda^*(s)$ for any $s\in[t, t + L^*(t)]$ (here * means that these functions may depend on the already simulated point patterns $\bx_{(\ra i)}$ and the part of $\bx_i$ in $L_i(0,t)$). The algorithm is as follows:
\begin{enumerate}
	\item Let $t=0$ and $j=0$
	\item Repeat:
	\begin{enumerate}
		\item Calculate $M^*(t)$ and $L^*(t)$.
		\item Generate (independently) $T\sim \mathrm{Exp}(M^*(t))$ and
		$U\sim \mathrm{Unif}([0,1])$
		\item If $t+T>|L_i|$, end repeat loop
		\item Else if $T>L^*(t)$, let $t=t+L^*(t)$
		\item Else if $U>\tilde\lambda^*(t+T)/M^*(t)$, let $t=t+T$
		\item Else, let $j=j+1$, $t=t+T$, and $x_j=u_i(t)$
	\end{enumerate}
	\item Output $(x_1,\ldots,x_j)$
\end{enumerate}

Both the inverse method and Ogata's modified simulation algorithm can be extended to the marked case by the following two modifications: First, insert an extra step such that each time a point has been simulated (and either kept for Ogata's modified thinning algorithm or moved for the inverse method), its mark should be simulated using the mark density $f^*$. Second, note that any function with a star may depend on both points and marks, not just points as in the unmarked case. 

The following proposition verifies that both of these algorithms produce a point process on a DALN with the correct distribution in both the unmarked and the marked case.
\begin{prop}
	Let $\omega\in\Omega$ and $L$ be a DALN, and produce point patterns on $L_{\omega(1)},\ldots, L_{\omega(N)}$  recursively 
	using either the inverse method or Ogata's modified simulation algorithm. Then the resulting simulation is a point process on $L$ with conditional intensity function $\lambda^*$.
\end{prop}

\begin{proof}
	Consider first the inverse method in the unmarked case used on a single line segment $L_i$. By Theorem 7.4.I in \cite{daley-vere-jones-03} the algorithm produces a simulation on $(0,|L_i|)$ with conditional intensity $\tilde\lambda^*$ when we consider $\Lambda^*(u_i(t))$ as a function of $t$. By the definition of
	$\lambda^*(u)=\tilde\lambda^*(t)$ for $u=u_i(t)$, we then get a correct simulation on $L_i$.
	
	Consider next Ogata's modified thinning algorithm in the unmarked case used on $L_i$. By \cite{ogata-81} and since $\lambda^*(u)=\tilde\lambda^*(t)$ for $u=u_i(t)$, a correct simulation is produced on $L_i$.
	
	Finally, since we follow the order given by $\omega$, $\bx_{(\ra i)}$ has always been simulated, when $\bx_{(i)}$ has to be simulated, and thus by the above argument each $\bx_{(i)}$ is simulated correctly, leading to a correct simulation of $\bx$ in the unmarked case.
	
	Turning to the marked case, we note that the above arguments still hold to prove that the points follow the correct distribution (where Proposition 7.4.IV in \cite{daley-vere-jones-03} needs to be used for inverse method, and the text accompanying Algorithm 7.5.V for Ogata's modified thinning algorithm). The proof is completed for the marked case by noting that the marks are always drawn from the correct distribution.
\end{proof}

\section{Residual analysis}\label{sec:residual_analysis}

One way of checking the fit of a model specified by a conditional intensity is to  calculate residuals and check their distribution. Consider first the unmarked case, and assume that we have observed a point pattern $(x_{(i),1},\ldots,x_{(i),n_i})$ on $L_i$ for every $L_i\in L$, and that we have obtained a fitted model with conditional intensity function $\hat\lambda^*$ for this dataset. Let
\[
\hat\Lambda^*(u) = \int_{L_i(0,t)} \hat\lambda^*(v)\dee\lambda_1(v) 
\]
where $u=u_i(t)$ and $t\in(0,|L_i|)$. We then calculate the residuals for the points on $L_i$ given by $(\hat\Lambda^*(x_{(i),1}),\ldots,\hat\Lambda^*(x_{(i),n_i}))$, which we (with a slight abuse of notation) denote by $\hat\Lambda^*(\bx_{(i)})$. If the model is correct, then, by Proposition~7.4.IV in \cite{daley-vere-jones-03}, $\hat\Lambda^*(\bx_{(i)})$ is a unit-rate Poisson process on the interval
$(0,\hat\Lambda^*(\overline{e}_i))$, and the residual processes $\hat\Lambda^*(\bx_{(i)})$ are independent for $i=1,\ldots,N$. 

In order to check the fit of the model given by $\hat\lambda^*$, we need to check whether the residuals form a Poisson process. The question is now how we best check this. One possibility is to check each process $\hat\Lambda^*(\bx_{(i)})$ separately, but there may be few points in each $\hat\Lambda^*(\bx_{(i)})$ and further, we lose information on any discrepancies around the junctions. For example, a particularly large gap around a junction may indicate a discrepancy between the model and the data, but since this gap will then be divided over several $\hat\Lambda^*(\bx_{(i)})$ it may be hard to discover. A better approach might be to construct a network with the same connecting junctions as the original network, and place the residuals on this network. The main complication is that the lengths of the line segments have changed from $|L_i|$ to $\hat\Lambda^*(|L_i|)$, so we cannot use the original network, and indeed the changed lengths may imply that there exists no network in $\RR^d$ with the correct line segments lengths and the correct connecting junctions. However, we can consider this network in a more abstract sense and apply any method for checking that a point pattern on a network comes from a unit-rate Poisson process, provided that the method does not rely on a correct Euclidean geometry of the network. 

One method for testing whether the residuals follow a unit-rate Poisson process model 
is to perform a global rank envelope test \citep{myllymaki-17} with the empirical geometrically corrected $K$-function or pair correlation function \citep{ang-etal-12} as test function. Note that this approach effectively ignores the directions present in the network (see Section~\ref{sec.ext} for further comments on including directions in the $K$-function).

Another method is based on interevent distances. To define these for a directed linear network, first recall that for point processes on the time line specified by a conditional intensity function, residual analysis often includes an investigation of the interevent times, i.e.\ the times between consecutive points of the residual process. If the proposed model is correct, the residuals constitute a unit-rate Poisson process which means that the interevent times are IID exponential variables with mean 1. In practice this can e.g.\ be checked visually by considering Q-Q-plots or histograms. For a point pattern $\bx$ on a DALN $L$, we can define a similar concept, the interevent distances, as the set 
\[
\{\ddl(x_{i},x_{j}):x_i,x_j\in\bx,\,x_i\in\pa(x_j)\},
\]
that is, the distance(s) to a point from its parent(s). If $L$ is an out-tree there is at most one parent for each point in $\bx$.  
For a unit-rate Poisson process on $L$, the interevent distances that corresponds to the distance between two consecutive points on the \textit{same} line segment are independent exponentially distributed variables with mean 1. However, interevent distances going across the same junction are not independent, since a part of the network is shared by the intervals corresponding to the interevent distance. One possible solution is to exclude all such interevent distances when comparing interevent distances to the exponential distribution,  but then information around the junctions is lost. Another solution is to consider all interevent distances and thus ignoring the dependency (which in practice may occur only for a small portion of the interevent distances depending on the number of junctions).

Generalizing this to the marked case in a sensible way is tricky. If we focus on the multivariate case, i.e.\ when the mark space is finite, some progress can be made. Assume that we have estimated the conditional intensity function by $\hat \lambda(\cdot, \cdot)$, and let 
\[
\hat\Lambda_m^*(u) = \int_{L_i(0,t)} \hat\lambda^*(v,m)
\dee\lambda_1(v)
\]
for a fixed mark $m \in \mathbb{M}$. Then all points with mark $m$ on a line segment $L_i$ is transformed using $\hat\Lambda_m^*$ to an interval $(0,\hat\Lambda_m^*(|L_i|))$. Note in particular that the intervals have different lengths, so the residuals for points with different marks end up in differently sized networks. By Proposition 7.4.VI.(a) in \cite{daley-vere-jones-03}, the residual processes thus obtained for different marks should behave like independent unit-rate Poisson processes provided the model is fitting well. We can apply the above techniques to each process separately to check whether these are unit-rate Poisson processes. Ideally we should also check whether each of these processes are independent of each other, but it seems to be hard to make a general test for this, since the processes are located on different networks.

The fact that points with different marks end up in different networks for the multivariate case hints at the difficulty in getting anything useful out of residual analysis for the general marked case. Proposition 7.4.VI.(b) in \cite{daley-vere-jones-03} can be used in this case, but if we for example have a continuous mark distribution, typically no marks are equal, so each residual point will end up in intervals of different lengths, and it is in no way obvious how to combine this into something useful for model checking.

\section{Models}\label{sec.models}

New models for point processes on a DALN specified by a conditional intensity function essentially boils down to giving a mathematical expression for the conditional intensity function. There is a rich selection of standard models for temporal point processes that can be
expressed using the conditional intensity function. The main problem in adapting them to the case of a DALN is dealing with the fact that at junctions the network may join several line segments and/or split into several line segments. We consider a few examples of models here.

\subsection{Poisson process}\label{sec.pois}

If the conditional intensity function $\lambda^*$ is a
deterministic non-negative valued measurable function on $L$, say $\lambda$, that does not depend on points further up the network, then we get a Poisson process on $L$ with intensity function $\lambda$. For this particular model, the point process does not depend on the directions and is equivalent to a Poisson process
specified on an undirected linear network \citep[see e.g.][]{ang-etal-12}.

For a homogeneous Poisson process on $L$ with constant intensity $\lambda$, the maximum likelihood estimate of $\lambda$ is simply $n/|L|$, where $n$ is the observed number of points (this follows from Proposition~\ref{prop.lik}).  

\subsection{Hawkes process}\label{sec.hawkes}

Another common temporal point process is the
Hawkes process or self-exciting process \citep{hawkes-71a,hawkes-71b,hawkes-72,hawkes-oakes-74}. We can extend it to a DALN using the conditional intensity function
\begin{equation}\label{eq.hawkes}
\lambda^*(u) = \mu + \alpha\sum_{x_i \in \bx : \, x_i\ra u} \gamma(\ddl(x_i,u)),
\end{equation}
where $\mu,\alpha>0$ are parameters and $\gamma$ is a density function
on $(0,\infty)$ that may depend on additional parameters. In the temporal case the model has the interpretation
that immigrants appear according to a Poisson process with intensity
$\mu$, then each immigrant, say $t_i$, produces a Poisson process of
offspring with intensity $\alpha\gamma(\cdot-t_i)$, and each offspring
produces another Poisson process of offspring, and so on \citep{moller-rasmussen-05,moller-rasmussen-06}.
In particular $\alpha$ can be interpreted as the mean
number of offsprings produced by each point. However, for a DALN containing a diverging junction, that is, a vertex with in-degree $= 1$ and out-degree $>1$, the offspring process is copied to
each outgoing direction, thus giving many more offsprings in mean. In
the case that there are multiple paths from $x_i$ to $u$ only the
shortest path count, meaning that clusters die out if they encounter
themselves further down the network.

If we want a version of the Hawkes process where clusters are split
equally when a diverging junction is met, and superposed when a
converging junction, i.e.\ a vertex with indegree $>1$ and outdegree $= 1$, is met, we can let
\begin{equation}\label{eq.modhaw}
\lambda^*(u) = \mu + \alpha\sum_{x_i \in \bx : \, x_i\ra u}\ \sum_{p\in P_{x_i\ra u}} g_p(x_i, u) \gamma(|p|),
\end{equation}
where $g_p(x_i,u)=1/n_p(x_i,u)$, and $n_p(x_i,u)$ is the product of the number of outgoing line segments met on each junction on the path $p$. For this model, $\alpha$ is the mean number of offspring resulting from each point (or more precisely, the mean number of offspring is less than or equal to $\alpha$ since the network is finite, and the
offspring processes thus get truncated). Using other functions $g_j$ may give other interpretations that are useful for various datasets.

For $u = u_i(t)$, the integrated conditional intensity for \eqref{eq.hawkes} is given by
\begin{align*}
\Lambda^*(u) = \mu t + \alpha \Bigl[\sum_{x_j \in \bx: \, x_j \to u} \Gamma(\ddl(x_{j}, u))  - \sum_{x_j \in \bx : \, x_j \in L^{\cup}_{\to i}}\Gamma(\ddl(x_j, \underline{e}_{i}))\Bigr],
\end{align*}
where $\Gamma$ is the distribution function associated with $\gamma$.

\begin{figure}
	\includegraphics[width=0.33\textwidth, trim={3cm 0 3cm 0}, clip]{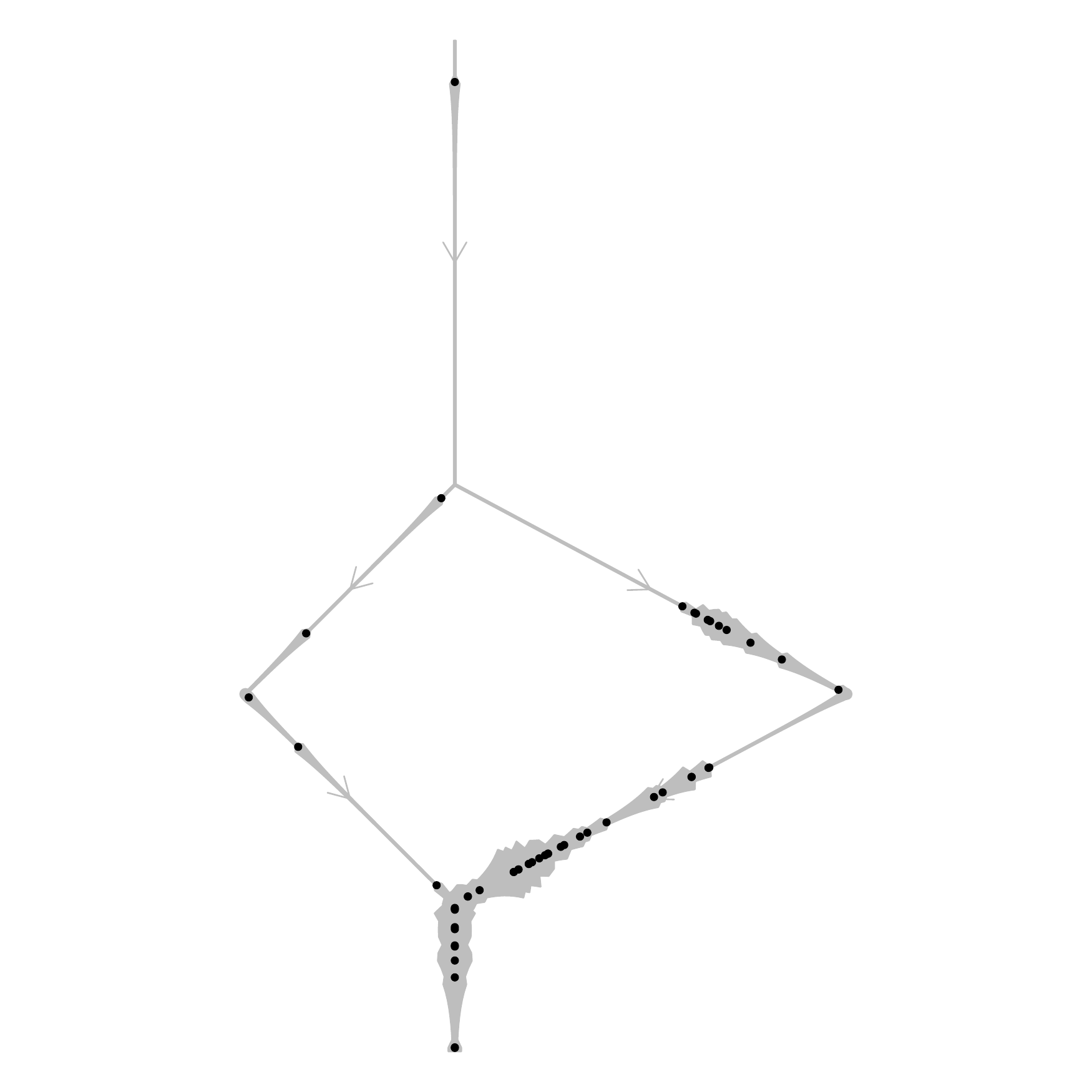}%
	\includegraphics[width=0.33\textwidth, trim={3cm 0 3cm 0}, clip]{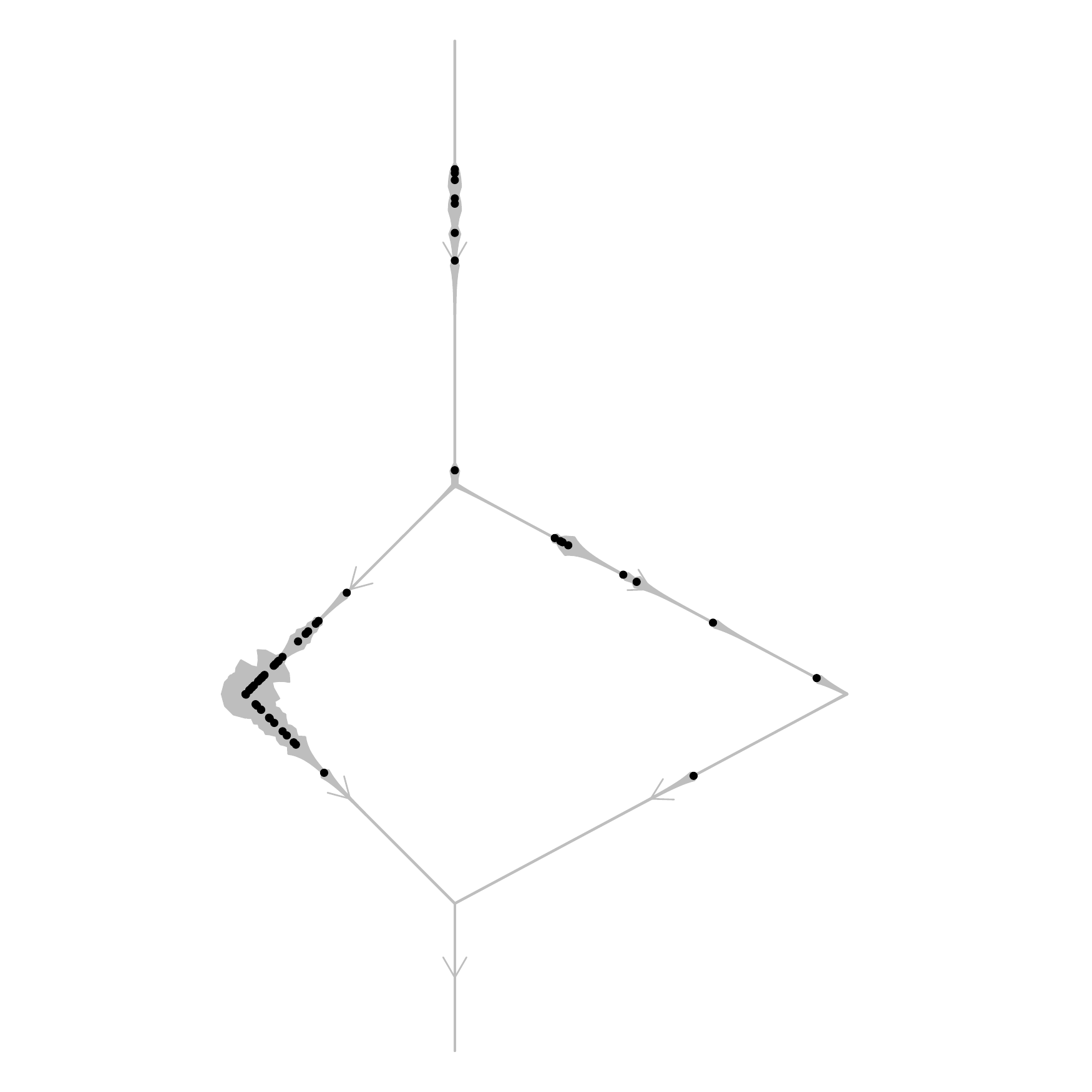}%
	\includegraphics[width=0.33\textwidth, trim={3cm 0 3cm 0}, clip]{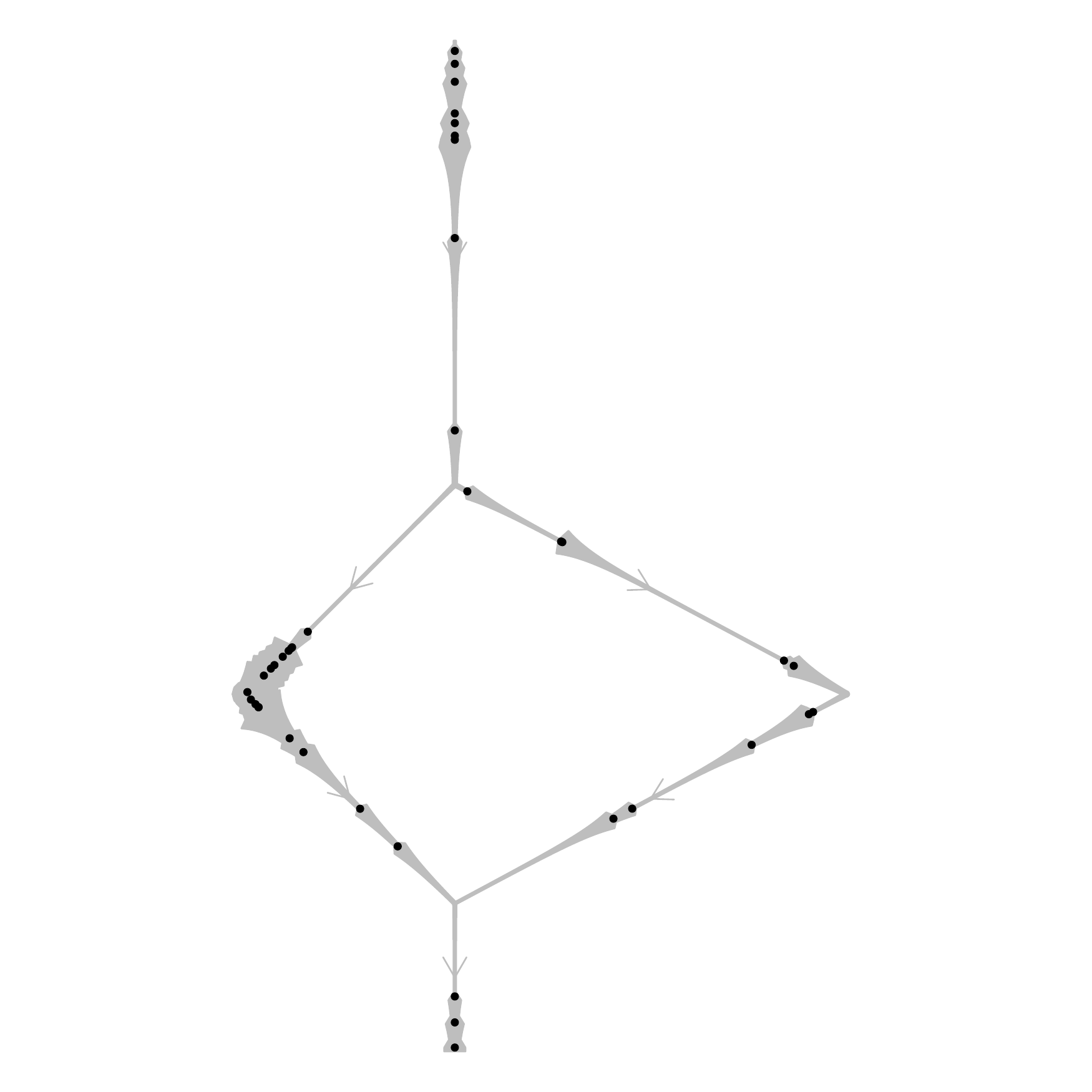}
	\caption{Simulations of Hawkes processes (on the DALN left of Figure~\ref{Fig2}) with parameters specified as follows. Left: $\mu = 1$, $\alpha = 0.8$ and $\gamma(t) = 5\exp(-5t)$. Middle: $\mu =1 $, $\alpha = 0.8$, and $\gamma(t) = 10\exp(-10t)$. Right: $\mu = 1$, $\alpha = 0.9$, and $\gamma(t) = 5\exp(-5t)$. Width of grey regions are proportional to the conditional intensity}
\end{figure}

\subsection{Non-linear Hawkes process}\label{sec.nlh}

A non-linear Hawkes process \citep{bremaud-massoulie-94,bremaud-massoulie-96} is obtained by inserting the conditional intensity function of the Hawkes process into a function $g:\mathbb{R}\rightarrow[0,\infty)$ such as the exponential function, that is,
\begin{equation}\label{eq.nonlinhaw}
\lambda^*(u) = \exp\Bigl[\mu + \alpha\sum_{x_i\in \bx : \, x_i\ra u}
\gamma(\ddl(x_i, u))\Bigr].
\end{equation}
The non-linear Hawkes process does not have a clustering and branching structure as the Hawkes process, so the modification done in \eqref{eq.modhaw} do not lead to any nice interpretations. On the other hand, if $\alpha>0$ the point process given by the conditional intensity function \eqref{eq.nonlinhaw} is clustered, if $\alpha<0$ it is regular, and if $\alpha=0$ it is a homogeneous Poisson process, so the model is rather flexible.

\subsection{Self-correcting process}\label{sec.selfcorr}
To model regular point patterns on DALNs, we further introduce a modification of the self-correcting process from the temporal setting \citep{isham-westcott-79}. The conditional intensity of a self-correcting process increases exponentially as the distance to the starting point increases, while it decreases whenever a point occurs. To adapt such a process to a DALN, we need to specify a meaningful starting point from which we measure distance. Therefore, we require that the DALN $L$ has a vertex $v_0$ such that $\ddl(v_0, u) < \infty$ for all $u \in L$. Note, for an out-tree, $v_0$ is simply the root of the tree, and for consistency we use the terminology root for $v_0$ even if the network is not an out-tree. 

Then we specify the self-correcting process by
\begin{align}\label{eq:cond_int_selfcorrecting}
\lambda^*(u) = \exp\left\{\mu\ddl(v_0, u) - \alpha |\bx \cap \mathrm{sp}(v_0,u)|\right\},
\end{align}
where $|\bx \cap \mathrm{sp}(v_0,u)|$ is the number of points from $\bx$ on the shortest directed path $\mathrm{sp}(v_0,u)={\arg\min}_{p\in P_{v_0\ra u}}|p|$ from $v_0$ to $u$, and $\mu, \alpha > 0$ are parameters controlling the overall intensity and the degree of repulsion. 
With this definition, only the points lying on the shortest directed path between  $v_0$ and $u$ affect $\lambda^*(u)$. For networks with paths of same length joining $u_1$ and $u_2$, the conditional intensity specified by \eqref{eq:cond_int_selfcorrecting} is somewhat ambiguous as $\mathrm{sp}(u_1, u_2)$ is not necessarily unique. An alternative, that may be more natural for some applications, is to count events on all paths from $v_0$ to $u$ (and not only on the shortest directed path). However, if $\bX$ is specified by \eqref{eq:cond_int_selfcorrecting}, the restriction of $\bX$ to $p_{v_0\ra u}$ for any $u \in L$ is a temporal self-correcting process on the interval $(0, |p_{v_0\ra u}|)$.

Another possible alteration of \eqref{eq:cond_int_selfcorrecting}, is to substitute the exponential function with some positive function $g$.

To obtain an expression for $\Lambda^*$, let $\{x^i_1, \ldots, x^i_{n_i}\} = \bx \cap L_i(0, t)$ denote the $n_i$ events falling on the partial line segment $L_i(0, t)$, while $x^i_0 = u_i(0)$ and $x^i_{n_i + 1} = u_i(t)$ denote the endpoints of $L_i(0, t)$. 
Then, for $u = u_i(t)$,
\[
\Lambda^*(u) =  \frac{c(\bx, i)}{\mu}\sum_{j = 0}^{n_i} \exp(-\alpha j) \left\{\exp\left[\mu \ddl(x_{0}^i, x_{j+1}^i)\right] - \exp\left[\mu  \ddl(x_{0}^i, x_{j}^i)\right]\right\},
\]
where $c(\bx, i) = \exp\left\{\mu \ddl(v_0, x_0^i) - \alpha |\bx \cap \mathrm{sp}(v_0,x_0^i)|\right\}$.

\begin{figure}
	\includegraphics[width=0.33\textwidth, trim={3cm 0 3cm 0}, clip]{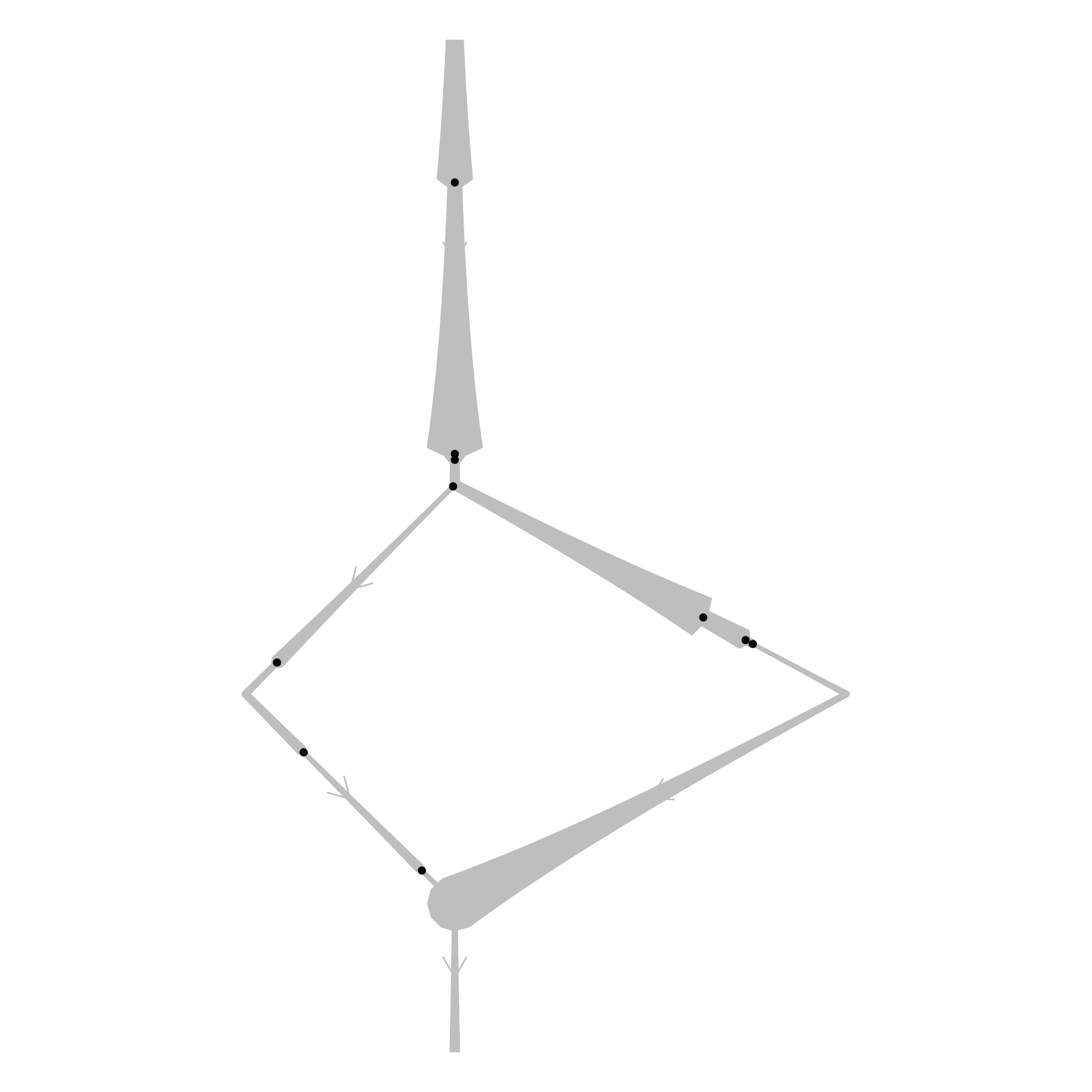}%
	\includegraphics[width=0.33\textwidth, trim={3cm 0 3cm 0}, clip]{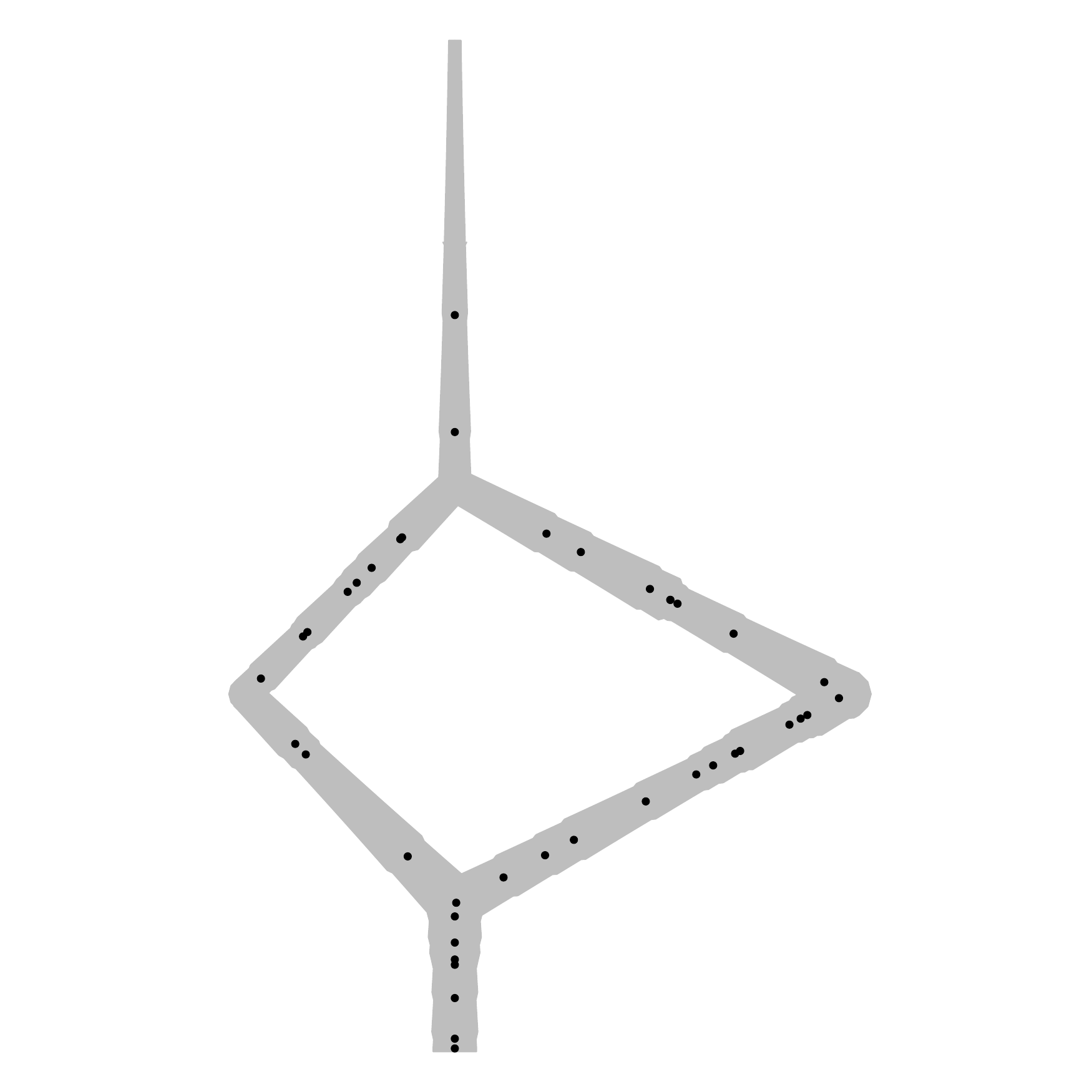}%
	\includegraphics[width=0.33\textwidth, trim={3cm 0 3cm 0}, clip]{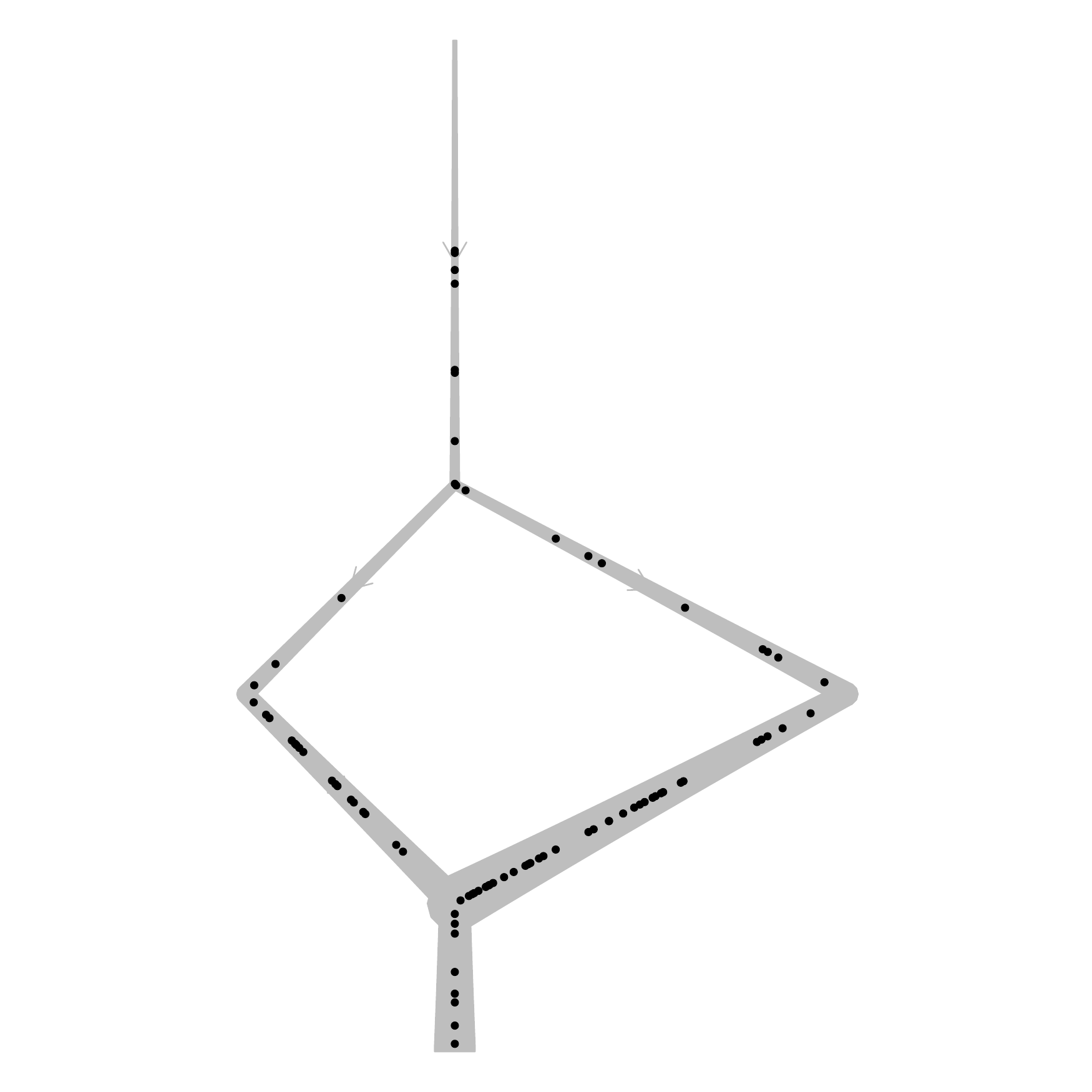}
	\caption{Simulations of self-correcting processes (on the DALN left of Figure~\ref{Fig2}) with parameters specified as follows. Left: $\mu = 0.8$ and $\alpha = 1$. Middle: $\mu = 0.4$ and $\alpha = 0.1$. Right: $\mu = 0.3$ and $\alpha = 0$}
\end{figure}

\subsection{Marked models}\label{sec.marked_models}

Any of the models in Sections~\ref{sec.pois}--\ref{sec.selfcorr} can be extended to the marked set up. The simplest case is to use so-called independent marks \citep[see e.g.][]{daley-vere-jones-03}, where the marks are independent of each other and independent of the points, with the sole exception that a mark is allowed to depend on the location of the point to which it is associated.
More interesting cases can be made by letting the conditional intensity at $u$ depend on the marks associated to the points in $L_{\ra u}$ (this is known as unpredictable marks if the other independence assumptions mentioned above still hold) and/or letting the mark associated to a point at $u$ depend on points on $L_{\ra u}$ and/or their associated marks. 

For an example of a marked point process, consider the Hawkes process given by \eqref{eq.hawkes} and assume that we are trying to model a dataset that has $K$ different types of points, denoted $1,\ldots,K$. Then the marked Hawkes process can be defined using the conditional intensity function 
\begin{equation}\label{eq.markedhawkes}
\lambda^*(u,m) = \mu_m + \sum_{x_i \in \bx: \, x_i\ra u} \alpha_{m_i,m} \gamma_{m_i,m}(\ddl(x_i,u)),
\end{equation}
where for $m,m'\in\{1,\ldots,K\}$ the parameters in the model are
given by $\mu_m, \alpha_{m,m'}>\nobreak 0$, and $\gamma_{m,m'}$ are
density functions on $(0,\infty)$. This generalization has a high
number of parameters, and for practical use assumptions that some of
these parameters are equal would typically be made.

Similarly, the other models presented in this paper can be extended to multitype cases or more general marked cases, and the primary difficulty is producing practically relevant models with nice interpretations and a reasonably low number of parameters. Obviously, what this is depends on the data at hand.

\section{Data Analysis}\label{sec.data}

\subsection{Simulated data}\label{sec.simdata}

To investigate properties of the maximum likelihood estimates for
parameters in the Hawkes and self-correcting model, we performed a
simulation study using the two DALNs shown in Figure~\ref{Fig2}. As
results for the two networks are very similar, we only present results
for the DALN to the left in Figure~\ref{Fig2}. In order to investigate
increasing-domain asymptotic properties, we increase the size of the
network seven times by $50\%$ each time and denote the resulting
networks by sizes $s=1,\ldots,7$. For each $s$ we simulate 1000 Hawkes
processes $\mu = 1$, $\alpha = 0.8$, and
$\gamma(t; \kappa) = \kappa\exp(-\kappa t)$, where $\kappa = 5$, and
1000 self-correcting processes with $\mu = 0.4$ and $\alpha = 0.1$
using the inverse method. For each simulation, the parameters have
been estimated (1) jointly, by numerically maximising the
log-likelihood simultaneously for all parameters, and (2) marginally,
by fixing all but one parameter at the true value and then numerically
maximising the log-likelihood with respect to the remaining parameter.

Figure~\ref{fig:results_Hawkes1_simulation_study} shows box plots of the joint estimates for the simulated Hawkes processes; these suggest that the maximum likelihood estimator of $(\alpha, \mu, \kappa)$ is consistent. Estimating the parameters marginally give similar results (not shown here) but with a slightly lower empirical variance.

For the simulated self-correcting processes, the joint estimates shown in Figure~\ref{fig:results_Selfcorrecting1_simulation_study}--\ref{fig:results_Selfcorrecting1_alphaVSmu} are clearly positively correlated, and both $\mu$ and $\alpha$ are grossly overestimated. This behaviour may be explained by the way $\mu$ and $\alpha$ influence the conditional intensity in \eqref{eq:cond_int_selfcorrecting}. Specifically, $\mu$ controls how much the conditional intensity increases as the distance to the root grows, while $\alpha$ determines how much the conditional intensity decreases when a new point is met. As more points will occur when the distance to the root grows, an increase in $\mu$ may to some extent be balanced out by an increase in $\alpha$. The ridge seen in Figure~\ref{fig:results_Selfcorrecting1_loglik}, displaying contours of the log-likelihood for one of the simulations, confirms that it may be hard to identify the true values of $\alpha$ and $\mu$ as the estimates will be chosen somewhere along that ridge. 
The marginal estimates, shown in first panel of Figure~\ref{fig:results_Selfcorrecting1_simulation_study}, are less extreme and on average closer to the true value. Specifically, fixing $\alpha$, the marginal estimates of $\mu$ seem unbiased, while fixing $\mu$ give positively biased estimates of $\alpha$ but with a smaller bias as the network grows. 

This short simulation study, indicates that the behaviour of the maximum likelihood estimates are quite model dependent, and thus it may be hard to say anything about the distribution of these in general. This is discussed further in Section~\ref{sec.ext}.

\begin{figure}
	\centering
	\includegraphics[width=0.4\textwidth]{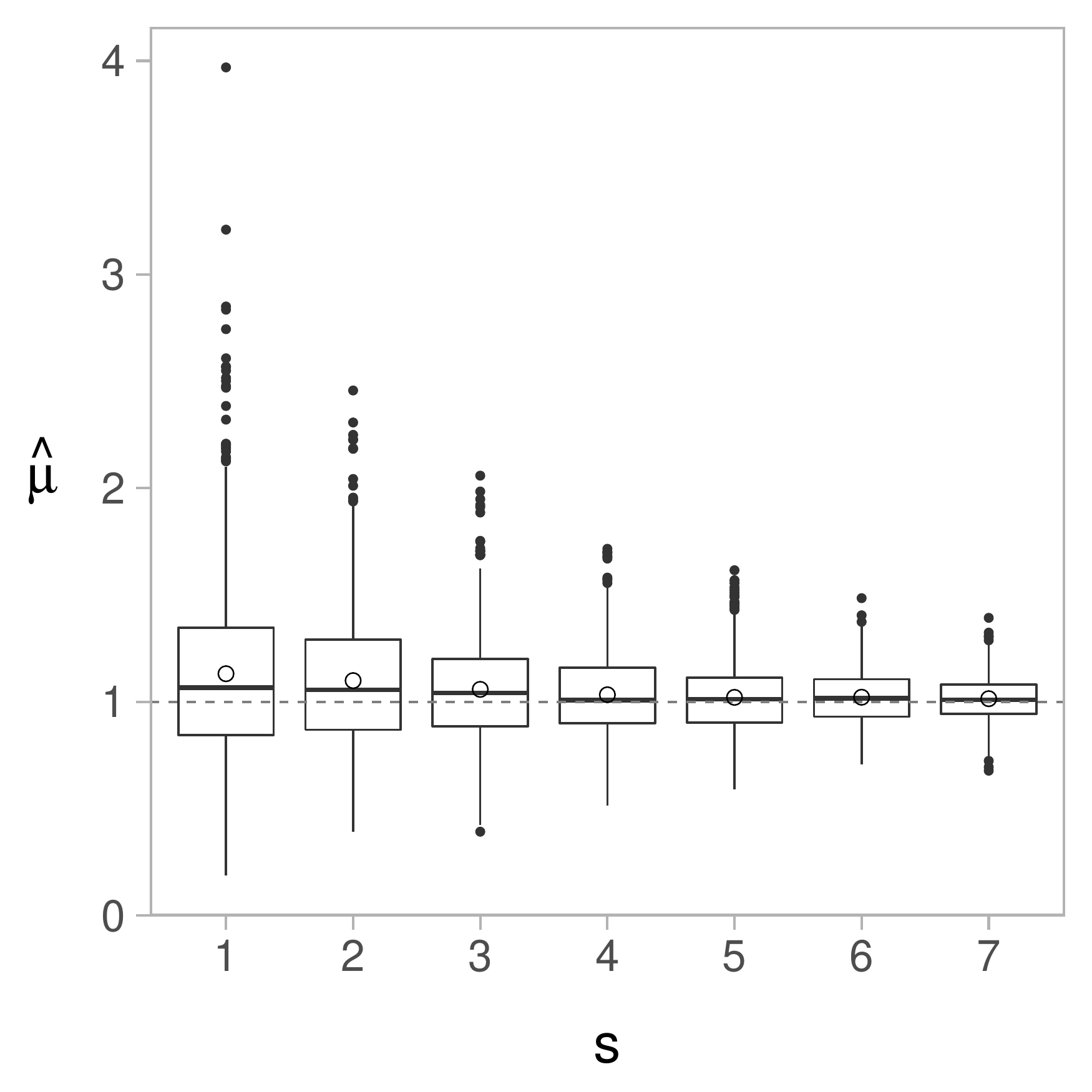}
	\includegraphics[width=0.4\textwidth]{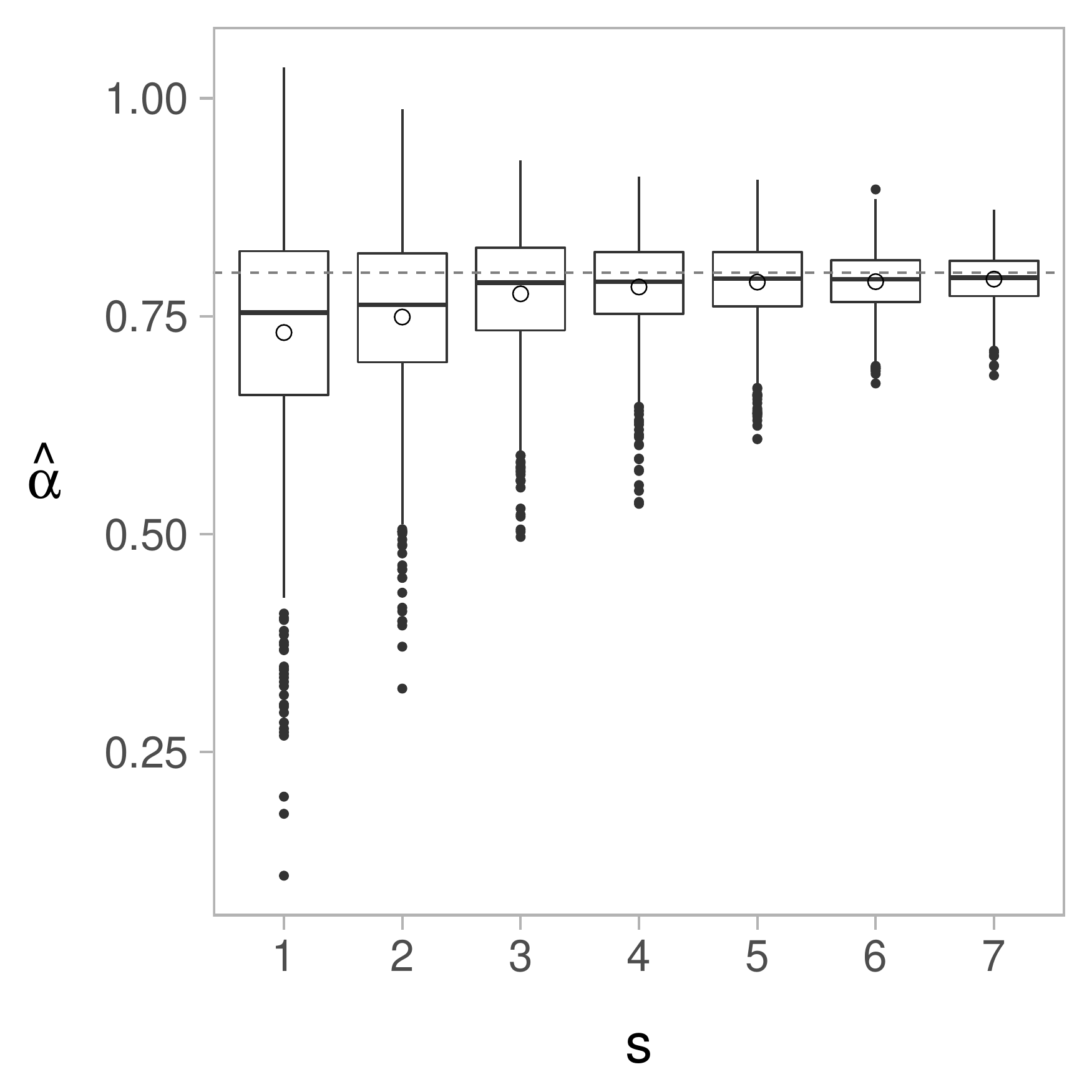}
	\includegraphics[width=0.4\textwidth]{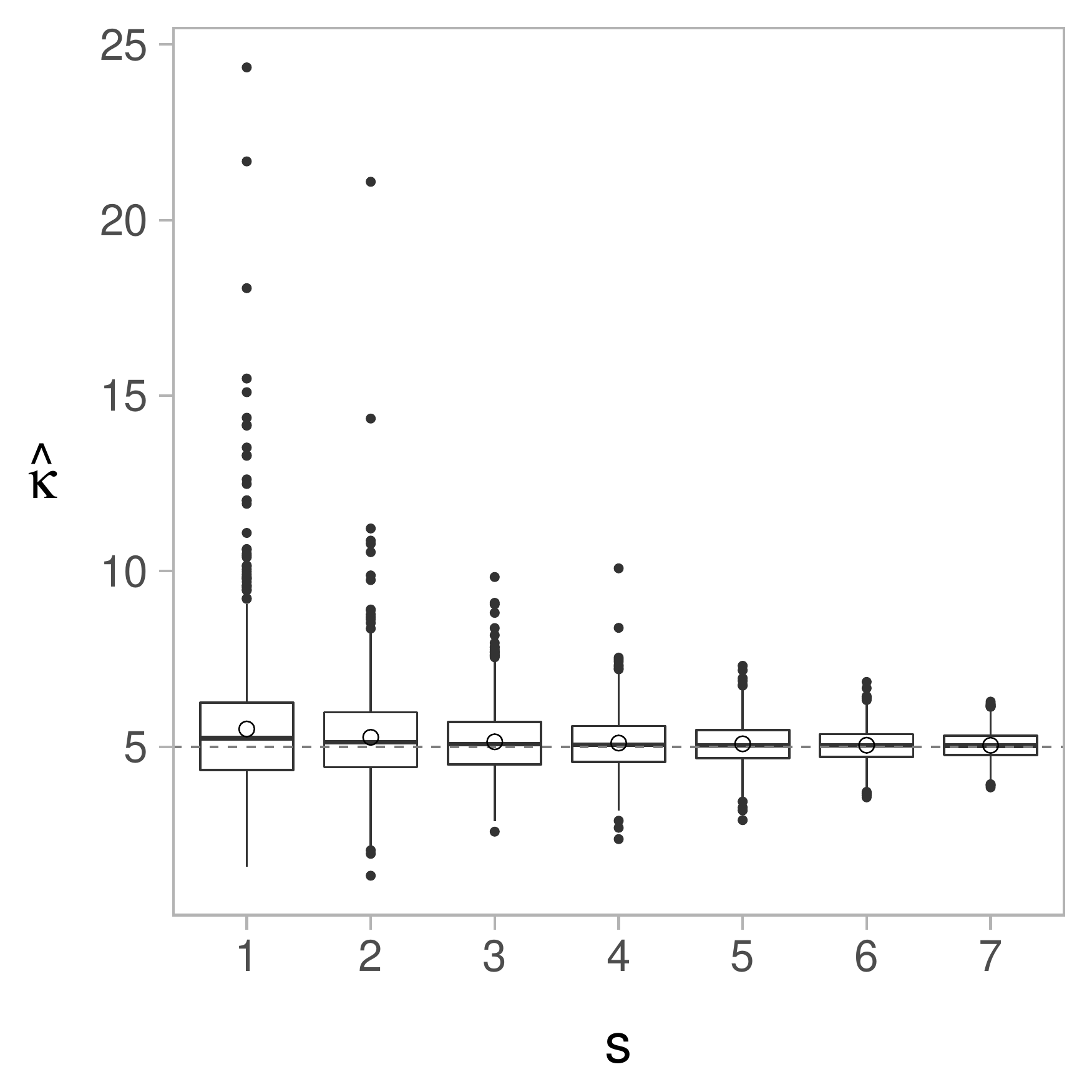}
	\caption{Results from simulation study: Box plot of joint parameter estimates, $\hat{\mu}$ (top left), $\hat{\alpha}$ (top right), and $\hat{\kappa}$ (bottom),  for the simulated Hawkes processes for each network of size $s$. Here $\circ$ is the empirical mean of the estimates}
	\label{fig:results_Hawkes1_simulation_study}
\end{figure}

\begin{figure}
\centering
\includegraphics[width=0.4\textwidth]{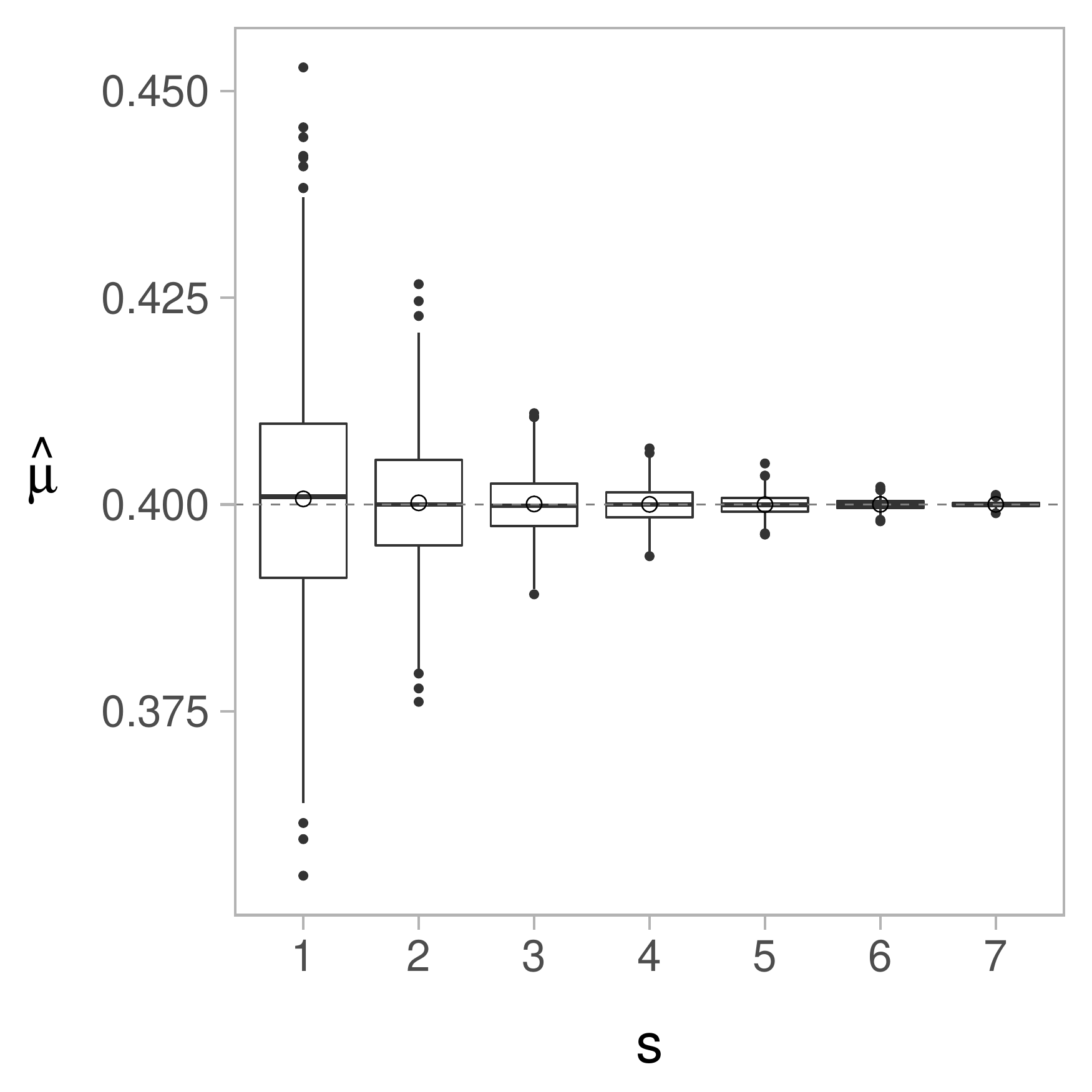}	
\includegraphics[width=0.4\textwidth]{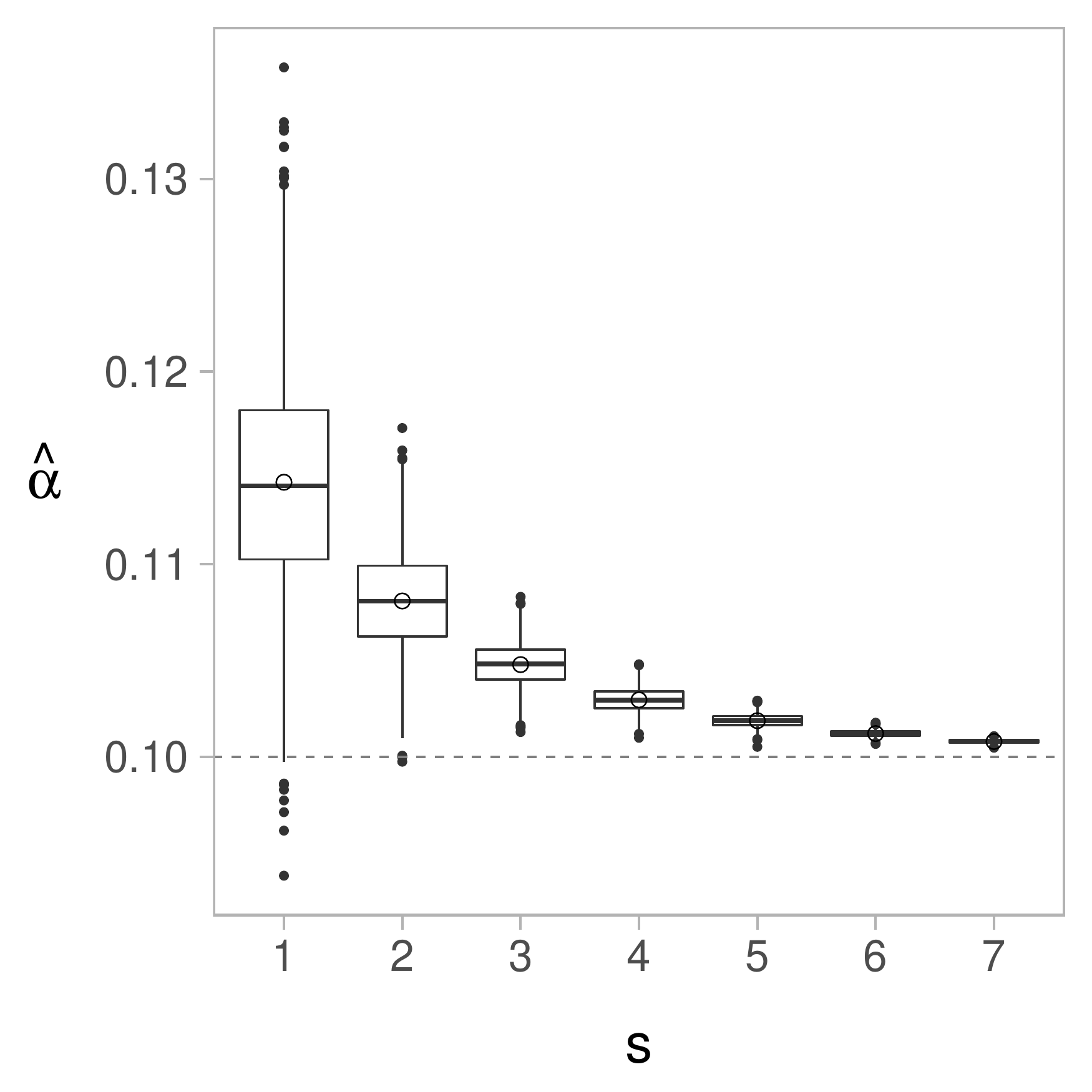}
\includegraphics[width=0.4\textwidth]{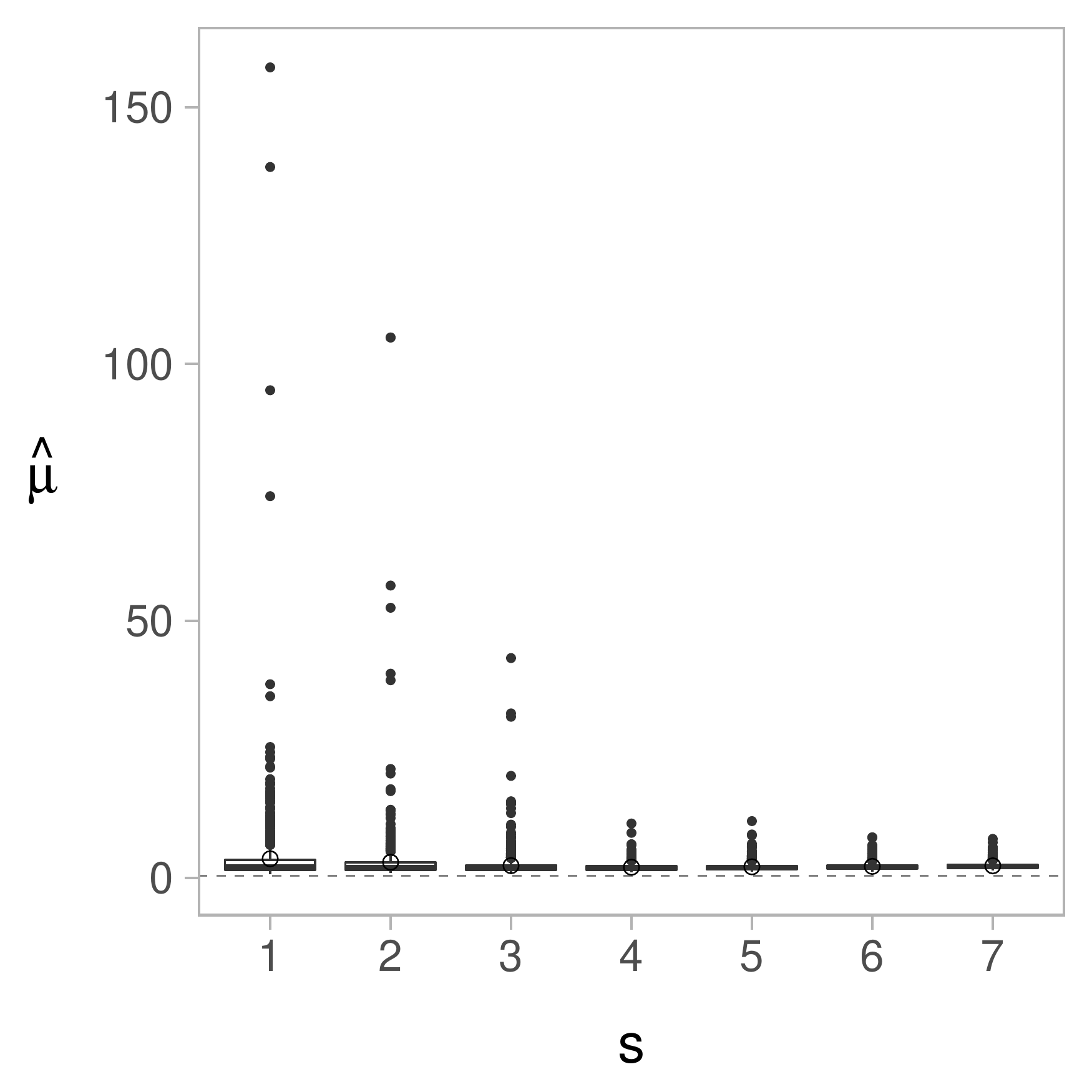}
\includegraphics[width=0.4\textwidth]{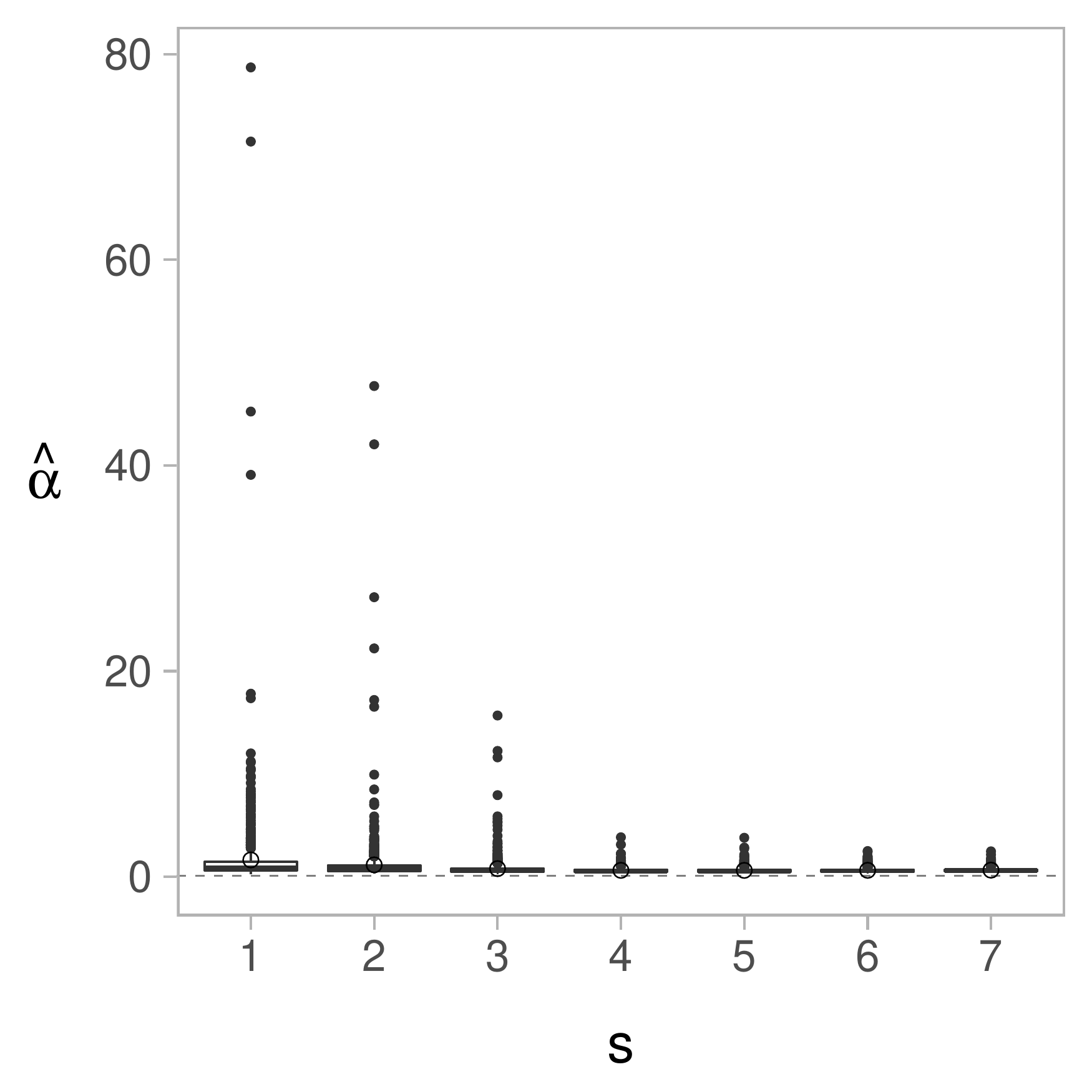}
\includegraphics[width=0.4\textwidth]{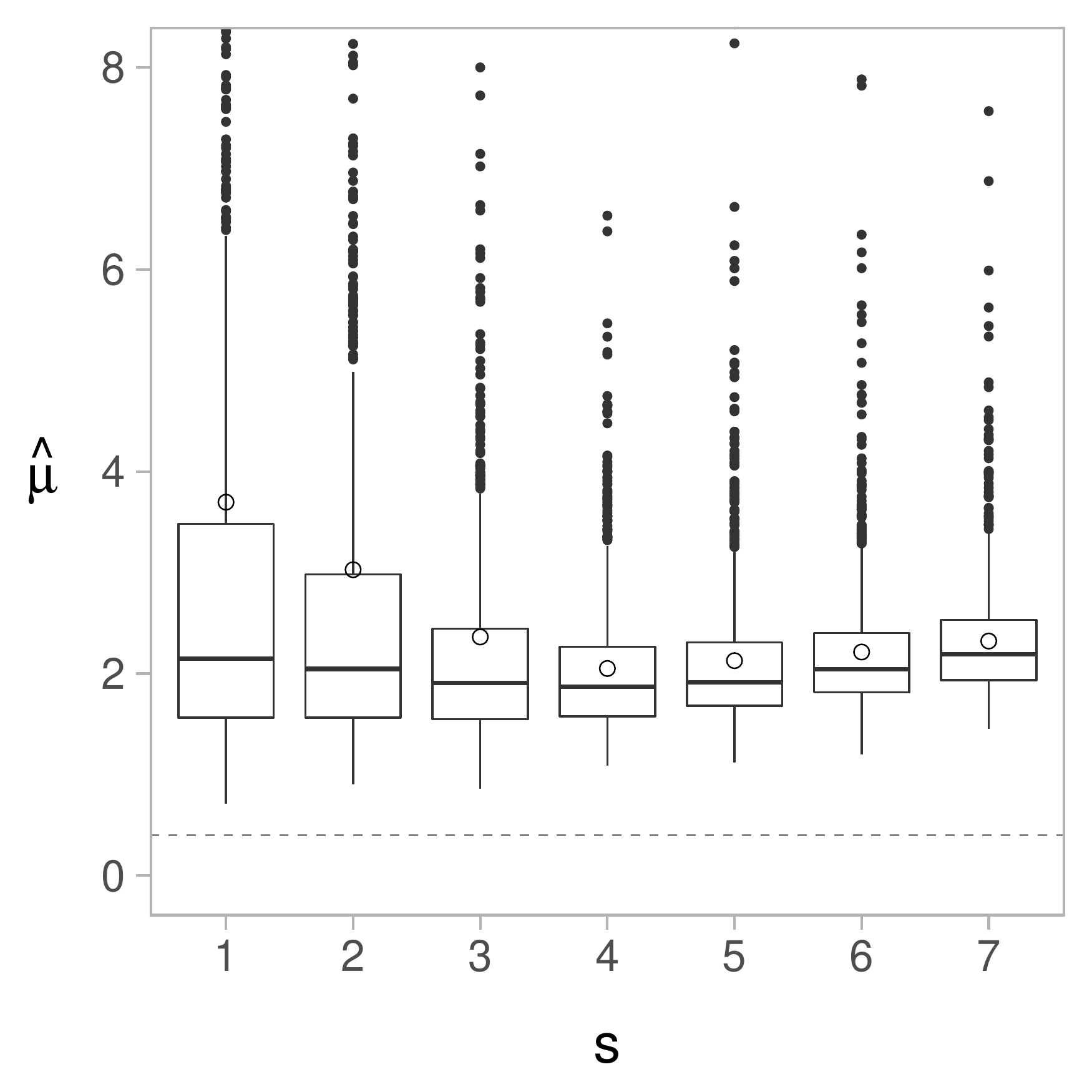}
\includegraphics[width=0.4\textwidth]{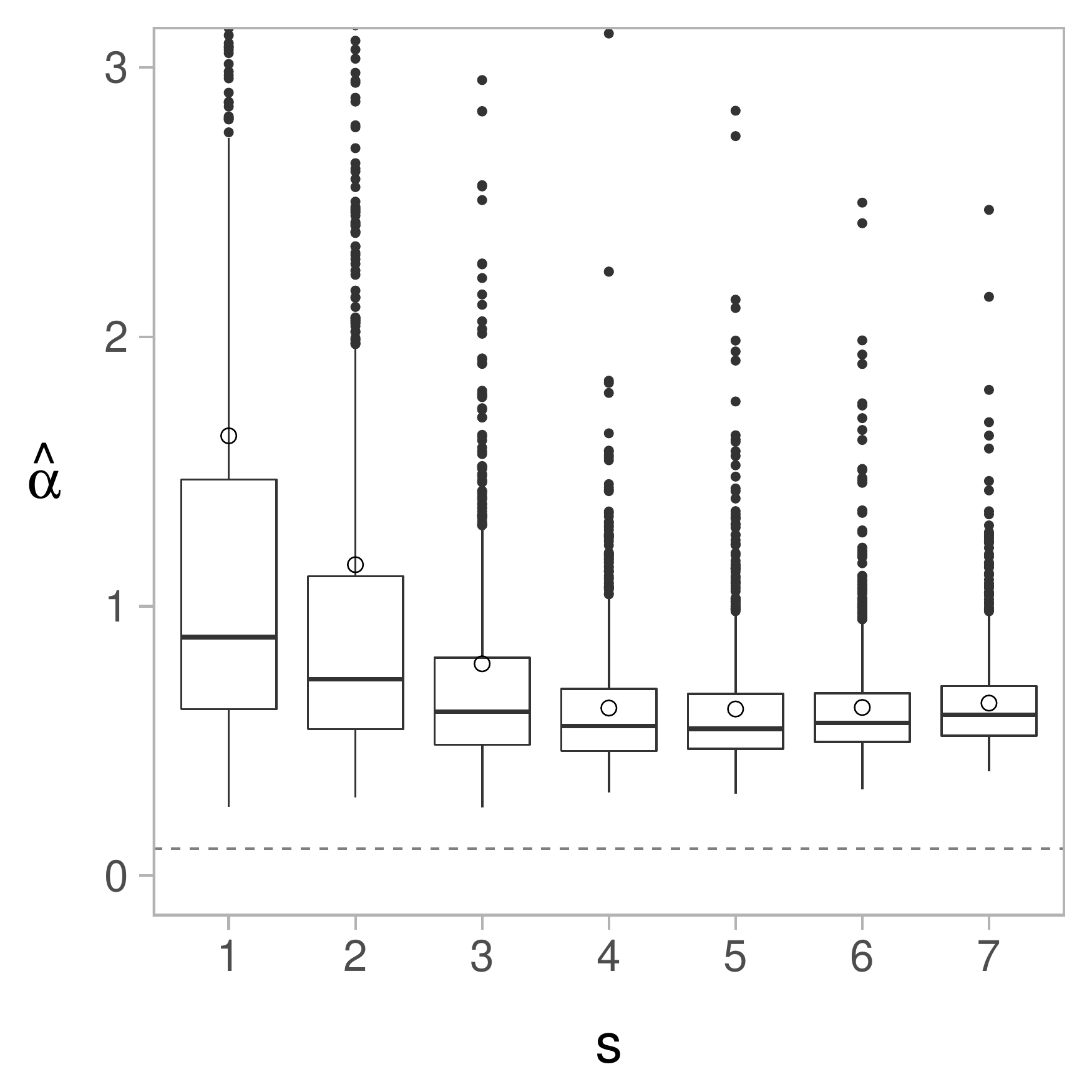}
\caption{Box plots of parameter estimates, $\hat{\mu}$ (left) and $\hat{\alpha}$ (right), for the self-correcting process on each network of size $s$. First panel: marginal parameter estimates. Second panel: joint estimates. Third panel: zoom of second panel.  Here $\circ$ denotes the mean estimate} 
\label{fig:results_Selfcorrecting1_simulation_study}
\end{figure}

\begin{figure}
	\centering
	\includegraphics[width=\textwidth]{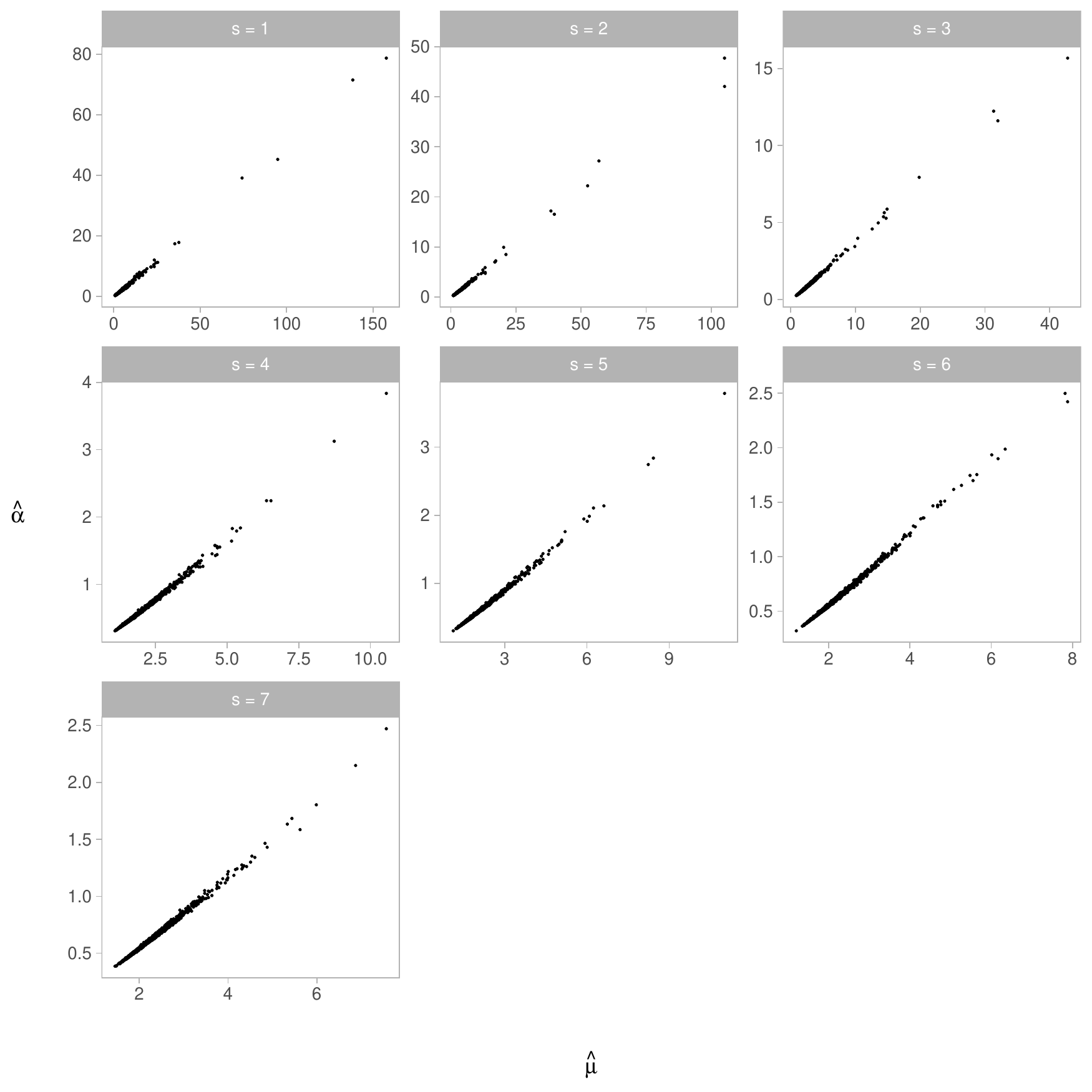}
	\caption{Plots of joint estimates from the self-correcting simulations. For each size of network $s$ considered in the simulation study, a plot of $\hat{\alpha}$ against $\hat{\mu}$}
	\label{fig:results_Selfcorrecting1_alphaVSmu}
\end{figure}

\begin{figure}
	\centering
	\includegraphics[width=0.75\textwidth]{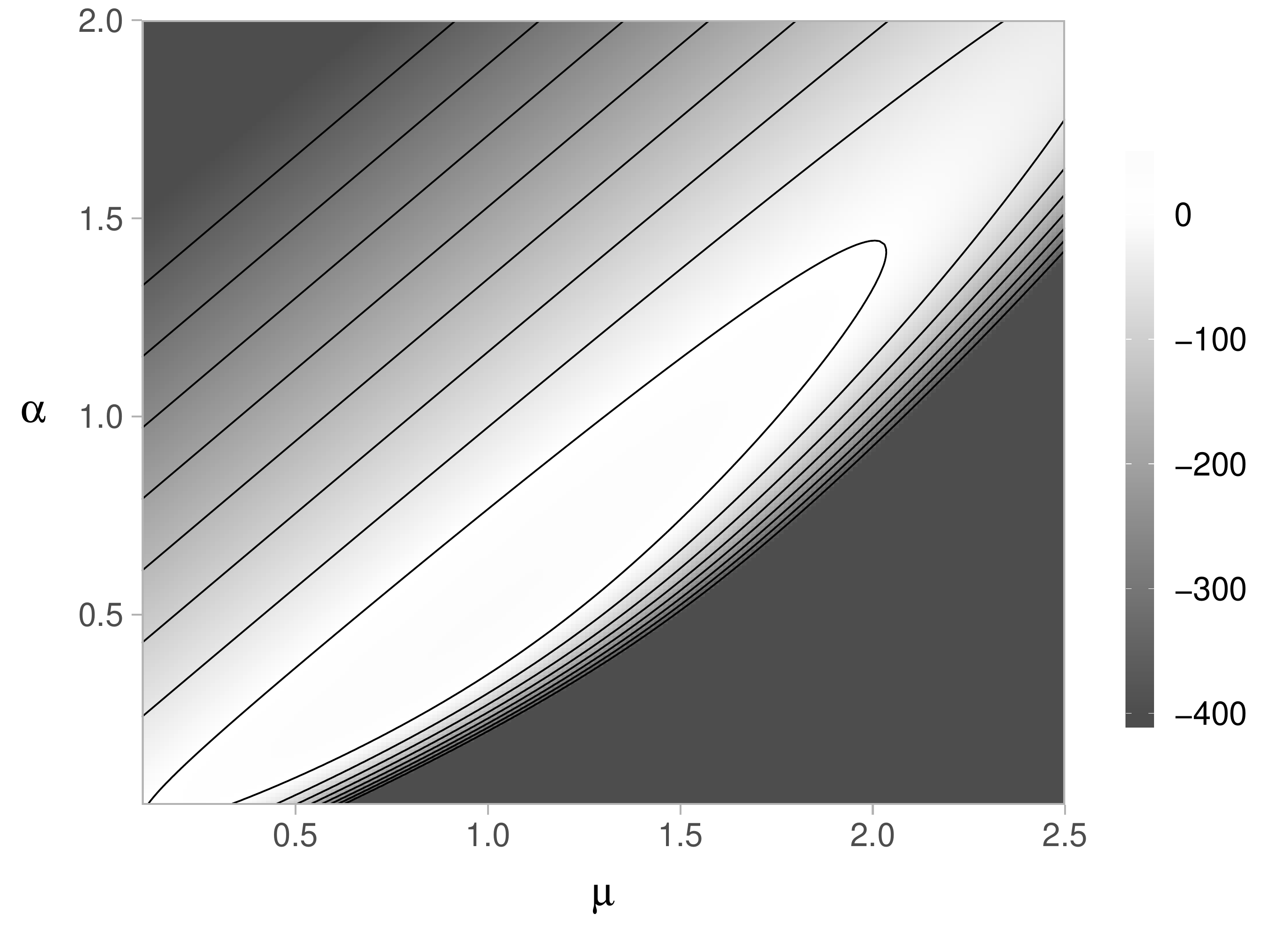}
	\caption{Contours of the log-likelihood for a simulated self-corrected process with $\mu = 0.4$ and $\alpha = 0.1$} 
	\label{fig:results_Selfcorrecting1_loglik}
\end{figure}

\subsection{Dendrite data}\label{sec.dendata}
In this section we consider a point pattern describing spine locations
on an apical dendrite tree from a mouse neuron. The dendrite tree was
first approximated by a linear network in $\RR^3$ (see
Figure~\ref{fig:neuron_network}). Next, a simplified version of the
network with fewer vertices was obtained by joining edges meeting at a
vertex of degree two. Then the network was embedded in $\RR^2$ (see
Figure~\ref{fig:neuron_network}) in order to directly use
functionalities from the R-package spatstat
\citep{baddeley-etal-15}. The embedding preserves distances, entailing
that distance-based analyses on the original network in $\RR^3$ and
the embedded network in $\RR^2$ are equivalent.  For example, the
geometrically corrected network $K$-function \citep{ang-etal-12} is
invariant under this kind of embedding. Letting the dendrite's
attachment point to the cell body be the root vertex of the network
tree, we can naturally consider the network as a DALN by introducing
directions going away from the root (hence the network thus obtained
is an out-tree).

First, we tested whether the spine locations can be described by a
homogeneous Poisson model with estimated intensity $n/|L|$, where
$n = 341$ is the number of spines and
$|L| = \SI{876}{\micro\meter}$
is the total network length. The empirical geometrically corrected
$K$-function (we return to the issue that this $K$-function ignores
directions in Section~\ref{sec.ext}), $\hat{K}$, may be used as a test
function in a global rank envelope test \citep{myllymaki-17}, where a
Monte Carlo approach is applied for approximating the distribution of
the test function under the null model. The global rank envelope
procedure both give critical bounds for the test function as well as
an interval going from the most liberal to the most conservative
$p$-value of the associated test. The $p$-interval associated with the
global rank envelope test for the homogeneous Poisson model is
$(0, 0.0096)$, indicating that the model is not appropriate.
Distances $r$, for which $\hat{K}(r)$ falls outside the critical
bounds (also called a global rank envelope) shown in
Figure~\ref{fig:neuron_envelope}, reveal possible reasons for
rejecting the model; in this case $\hat{K}(r)$ falls above the
envelope for $r$-values up to
$\SI{\approx 50}{\micro\meter}$, indicating
clustering at this scale.

To model the clustering, we next consider the dendrite tree as a
directed network and fit a Hawkes model, where we let $\gamma$ in
\eqref{eq.hawkes} be the density of an exponential distribution with
parameter $\kappa$. The three parameters, $\mu$, $\alpha$ and
$\kappa$, are estimated by numerically optimizing the
log-likelihood. The resulting estimates are $\hat{\mu} = 0.11$,
$\hat{\alpha} = 0.84$, and $\hat{\kappa} = 0.073$. According to
Section~\ref{sec:residual_analysis}, we can check whether the model
adequately describe our data by looking at the residuals. Again, we
use the global rank envelope procedure with $\hat{K}$ as test
function, but now for testing whether the residuals follow a unit-rate
Poisson model on the transformed network. The resulting $95 \%$-global
rank envelope, shown in Figure~\ref{fig:neuron_envelope}, has an
associated $p$-interval of $(0, 0.0068)$. However, the only
discordance detected between the residuals and the unit-rate Poisson
model with the global rank envelope is for $r$-values less than
$\SI{1}{\micro\meter}$. This may indicate that there is a small-scale
repulsion between the spines, which is not accounted for in the Hawkes
model.

In Figure~\ref{fig:interevent_times}, a Q-Q-plot of all interevent
distances in the residual process is shown along with labels
indicating whether the interevent distance is across a junction or
not. Regardless of whether we include these crossing interevent
distances or not, the distribution of the interevent distances seems
to deviate only slightly from the exponential distribution with mean
1.

\begin{figure}
	\begin{minipage}{0.5\textwidth}
		\includegraphics[width=\textwidth, trim={0 1cm 0 1cm}, clip,]{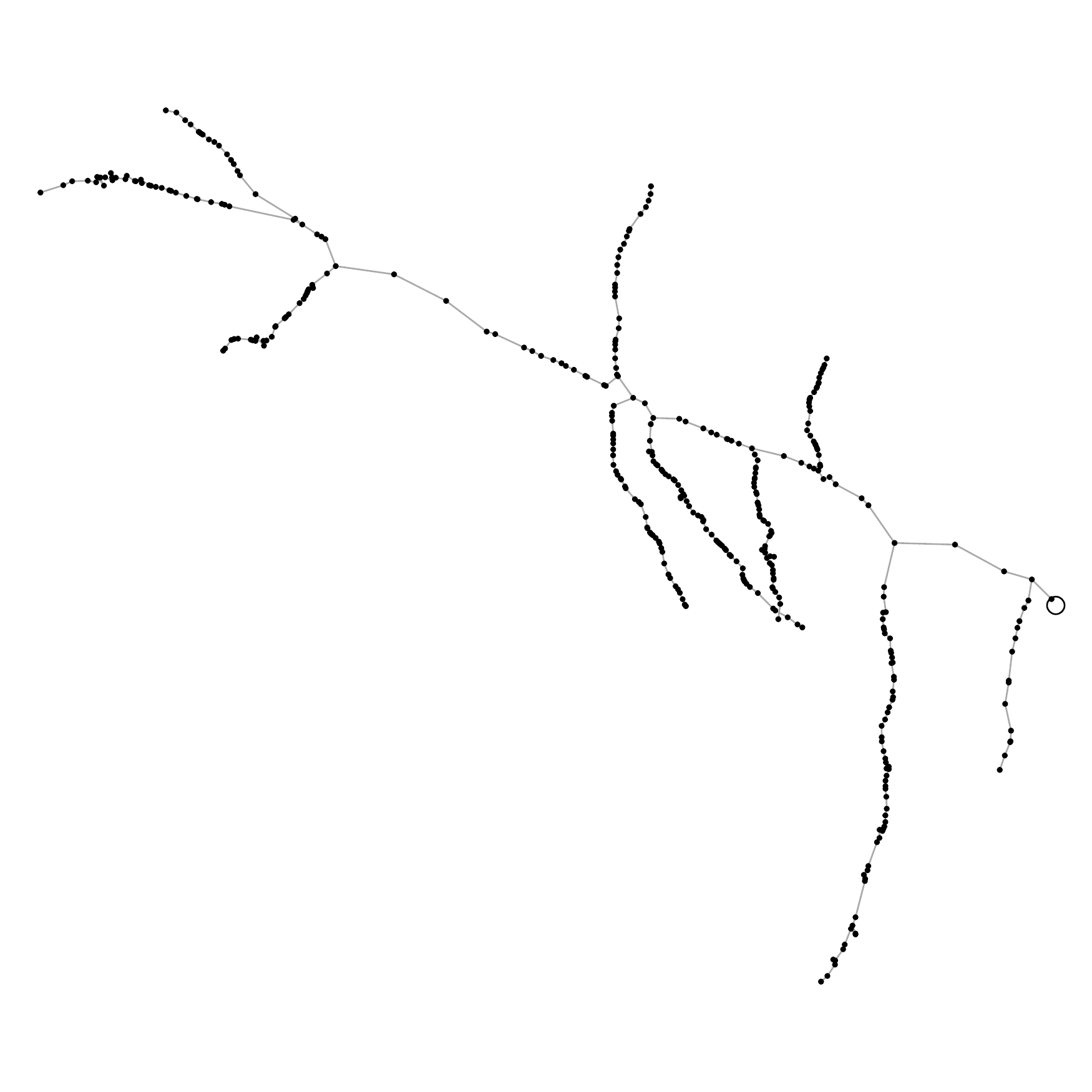}
	\end{minipage}\hspace{5mm}
	\begin{minipage}{0.5\textwidth}
		\includegraphics[width=\textwidth, trim={1cm 0 1cm 0}, clip]{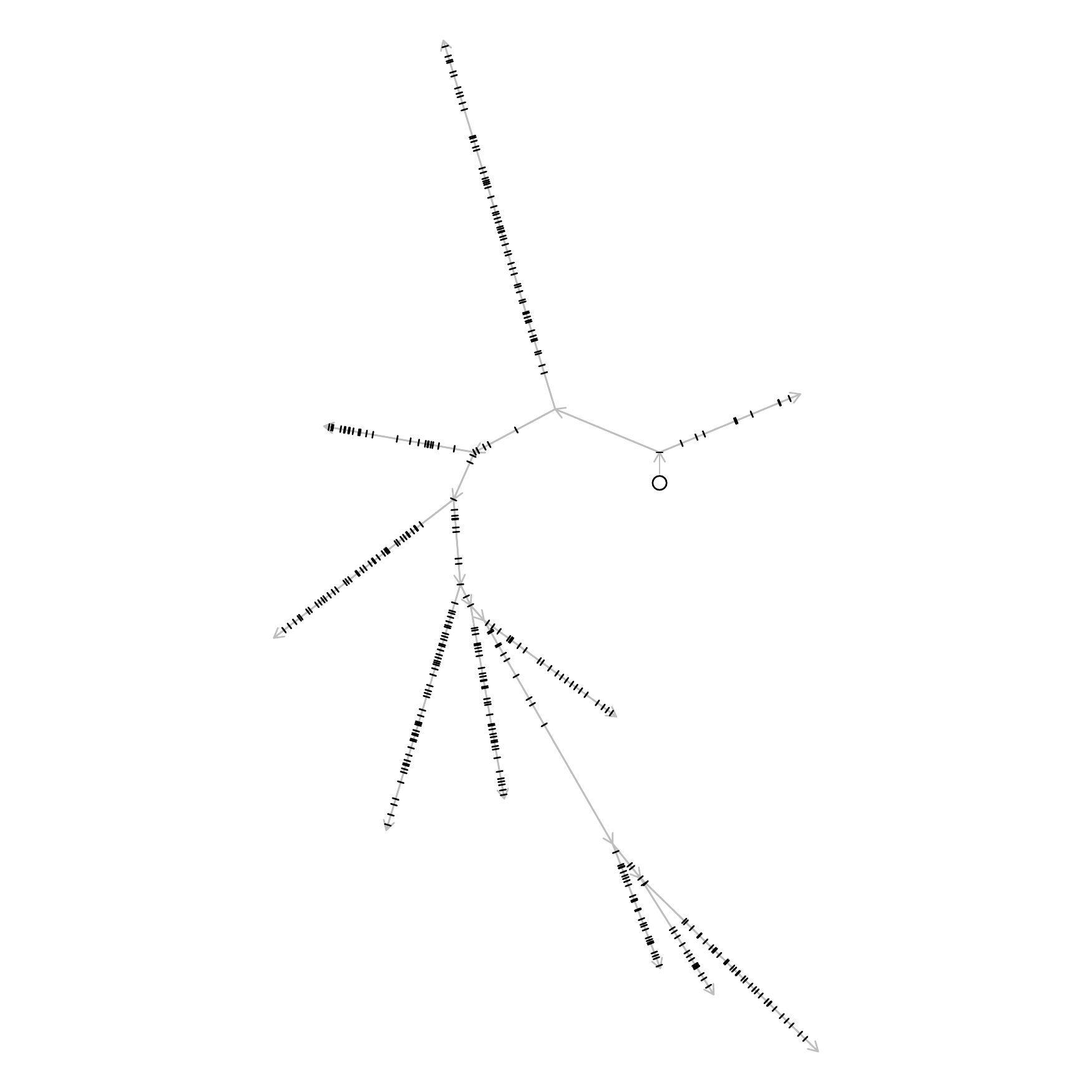}
	\end{minipage}
	\caption{Left: projection of the approximated dendrite tree onto $\RR^2$. Right: a distance-preserving embedding of the network into $\mathbb{R}^2$. Here $\circ$ identify the root of the dendrite tree, while $\bullet$ (left) and  \rule{1pt}{1.5ex} (right) are spine locations}
	\label{fig:neuron_network}
\end{figure}

\begin{figure}
	\centering
	\includegraphics[width=0.45\textwidth]{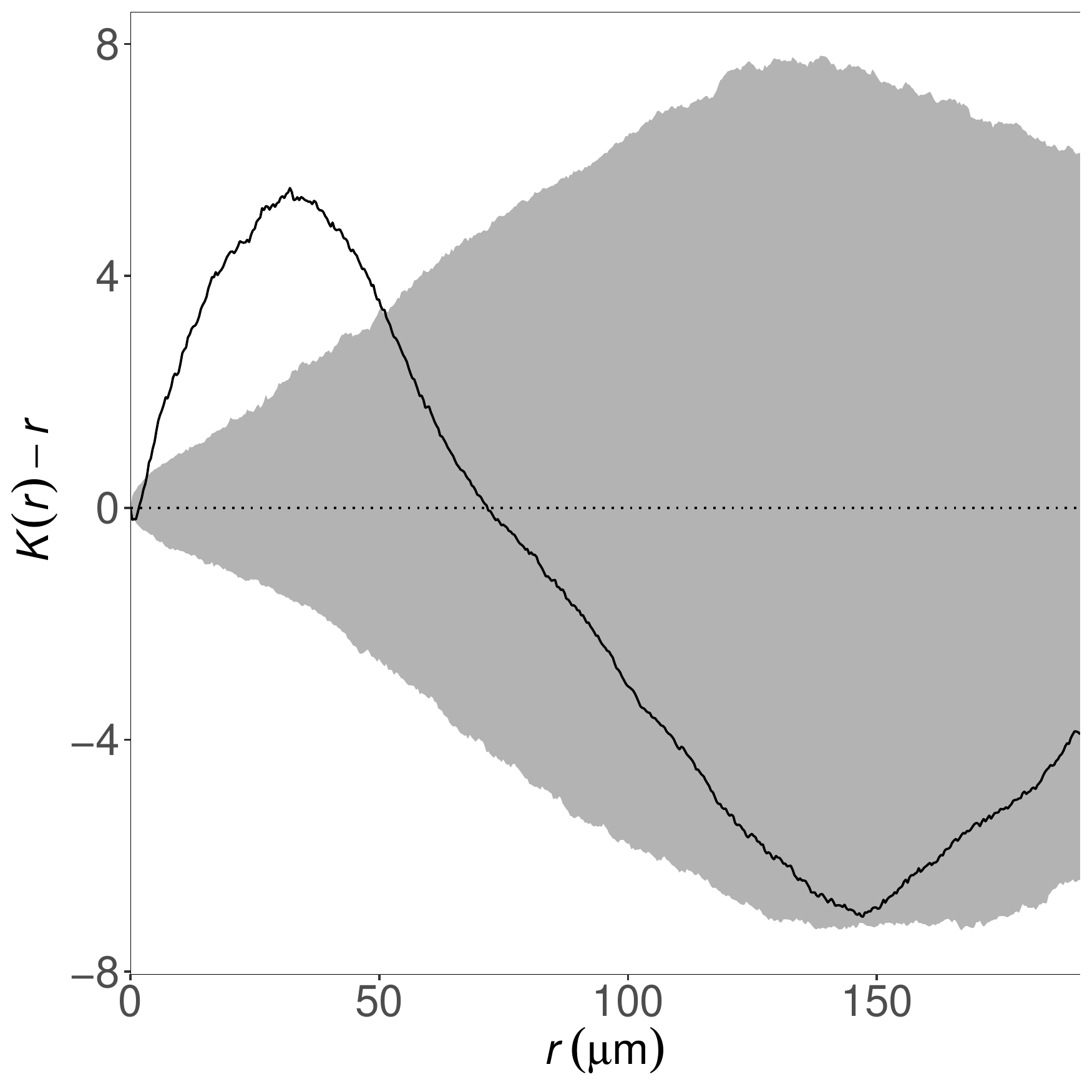}
	\includegraphics[width=0.45\textwidth]{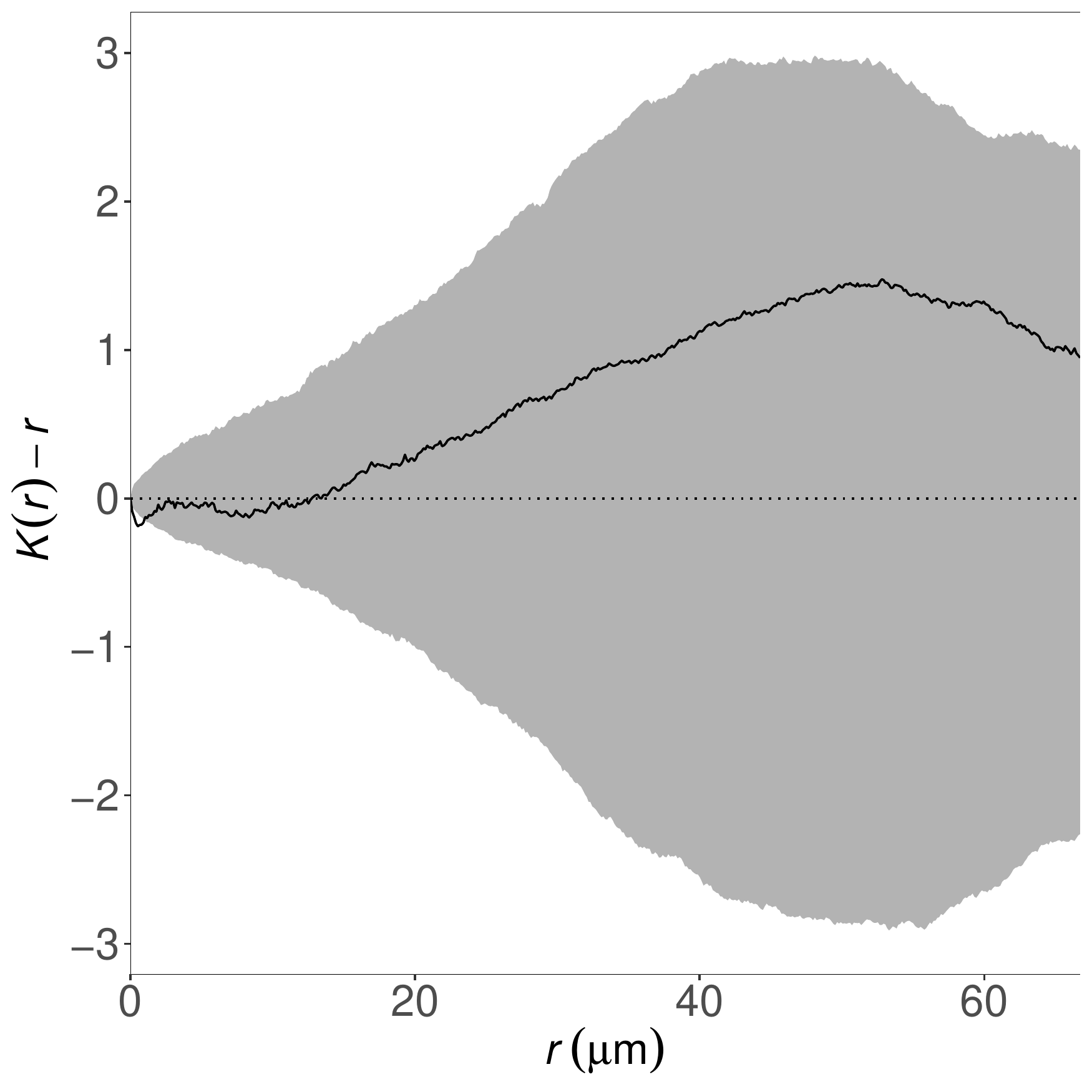}
	\caption{
		$\hat{K}(r)  - r$ (solid curve) for the observed spine locations (left) and the residuals from the fitted Hawkes model (right). The grey regions are $95 \%$-global rank envelopes based on 4999 simulations from a homogeneous Poisson model with intensity $n/|L|$ (left) and from a unit-rate Poisson (right), respectively}
	\label{fig:neuron_envelope}
\end{figure}

\begin{figure}
	\centering
	\includegraphics[width = 0.5\textwidth]{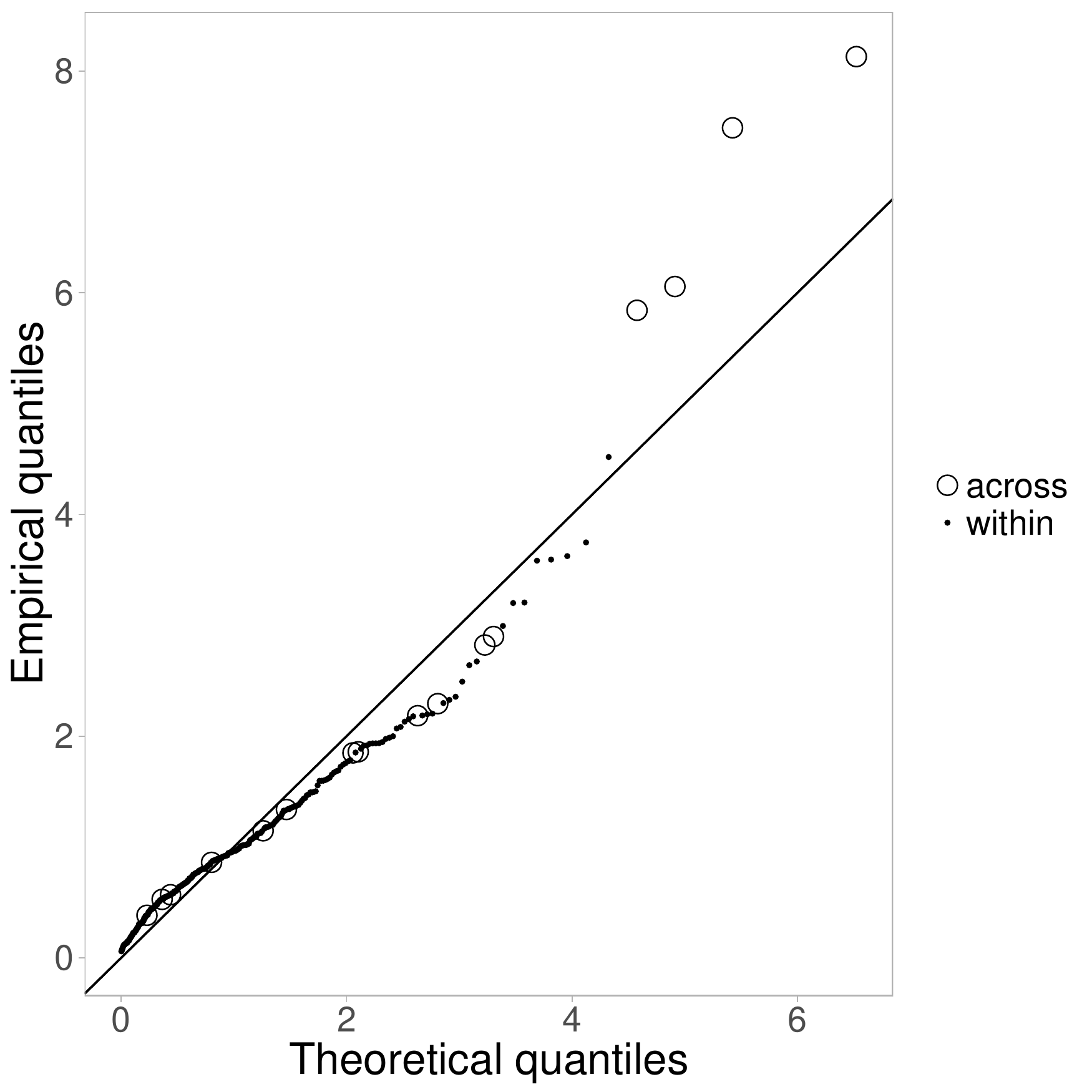}
	\caption{Q-Q-plot of quantiles for the interevent distances in
          the residual process for the spine locations vs.\
          theoretical quantiles from the exponential distribution with
          mean~1. The labels \textit{across} and \textit{within} refer
          to interevent distances for points on different line
          segments or on the same line segment, respectively}
	\label{fig:interevent_times}
\end{figure}

\section{Extensions and future research}\label{sec.ext}

In this final section we consider extensions and modifications of the models and methods as well as future research directions.

Recall that the so-called geometrically corrected $K$-function is introduced in \cite{ang-etal-12} for analysing point processes on linear networks, and among other things, this can be used for analysing a point pattern dataset for example to determine whether the points are more clustered or more regular than a homogeneous Poisson process. This information is relevant for choosing an appropriate model for modelling a dataset. For example, the Hawkes process in Section~\ref{sec.hawkes} is more clustered than a homogeneous Poisson process, the self-correcting process in Section~\ref{sec.selfcorr} is more regular, and the non-linear Hawkes process in Section~\ref{sec.nlh} can be both.
We used this in Section~\ref{sec.dendata} for checking whether the residual process behaved like a Poisson process as is expected if an adequately well-fitting model has been used for modelling the data. However, for a directed linear network the dependence structure is completely changed, and as a consequence the appropriate concepts of clustering and regularity are also changed. While the (undirected) geometrically corrected $K$-function certainly gives a good idea of the amount of clustering and regularity, the development of a directed geometrically corrected $K$-function is useful for quantifying such concepts in a more appropriate manner. We leave this as an object of future research. 

Essentially a linear network consists of a superposition of line segments in $\mathbb{R}^d$, but in \cite{anderes-etal-17} they have been generalized to graphs with Euclidean edges, which extends the linear network in various ways to include curve segments, crossing (but unconnected) segments, etc.\ Such a generalization can rather easily be made to directed linear networks to obtain a directed version of graphs with Euclidean edges, and all results in the present paper immediately extends to this case (we have only avoided making this extension to avoid a more cumbersome notation in this paper).

Furthermore, \cite{anderes-etal-17} also consider two different metrics on graphs with Euclidean edges: the shortest path metric, i.e.\ the length of the shortest path along the edges of the graph, and the resistance metric, i.e.\ the metric corresponding to the resistance in an electrical network (see also \cite{rakshit-etal-17} for use of various metrics on linear networks). We note that the quasi-metric $\ddl$ is the natural directed counterpart of the shortest path metric, and it would be the natural choice for modelling many kinds of point pattern data on a directed linear network. However, any other quasi-metric on $L$ can be used as a basis for building models on directed linear networks and may be relevant for practical applications where $\ddl$ is not appropriate.

The results in Section~\ref{sec.simdata} suggest that at least in some cases the maximum likelihood estimator has nice asymptotic properties. Specifically, the maximum likelihood estimates for the Hawkes processes seem unbiased and consistent, while the estimates in the self-correcting process are strongly biased and correlated. 
A proper development of asymptotic theory in the line of \cite{ogata-78} and \cite{rathbun-94} is important to establish properties of the maximum likelihood estimator for point processes specified by a conditional intensity function on a directed linear network.

\newpage

Section~\ref{sec.dendata} presents a very short analysis of a dendrite dataset using the conditional intensity function to build a point proces model on a directed linear network. The main purpose of this is to illustrate that the models and methods can be applied to real data. As a future research direction we plan to make a much more thorough analysis of the presented dendrite dataset and other similar datasets, where we will also model the spine types as marks, and derive practical results from the models with biological relevance.

\paragraph*{Acknowledgements.}
The authors would like to thank Abdel-Rahman Al-Absi who collected the
dendrite data, and Adrian Baddeley, Gopalan Nair and Valerie Isham for
preliminary discussions on ideas for the theory in the paper.

This work was supported by The Danish Council for Independent Research
-- Natural Sciences, grant DFF -- 7014-00074 ``Statistics for point
processes in space and beyond'', by the ``Centre for Stochastic
Geometry and Advanced Bioimaging'', funded by grant 8721 from the
Villum Foundation, and by the Australian Research Council grant
DP130102322, ``Statistical methodology for events on a network, with
application to road safety''.

\bibliographystyle{spbasic}      
\bibliography{bibliography}   

\end{document}